\documentclass{amsart}
\usepackage{amsmath,amssymb,amsthm}
\usepackage{pstricks}

\newtheorem{theorem}{Theorem}[section]
\newtheorem{lemma}[theorem]{Lemma}
\newtheorem{claim}[theorem]{Claim}

\newtheorem{proposition}[theorem]{Proposition}
\newtheorem{corollary}[theorem]{Corollary}
\theoremstyle{definition}
\newtheorem{example}[theorem]{Example}
\newtheorem{remark}[theorem]{Remark}

\def\en{\mathbb N}
\def\er{\mathbb R}
\def\C{\mathcal C}

\newcommand{\uv}[1]{\textquotedblleft #1\textquotedblright}

\begin{document}

\title[Notes on the trace problem]{Notes on the trace problem for separately convex functions}
\author{Ond\v rej Kurka \and Du\v san Pokorn\'y}
\thanks{The authors are junior researchers in the University Centre for Mathematical Modelling, Applied Analysis and Computational Mathematics (MathMAC). The first-named author is a member of the Ne\v{c}as Center for Mathematical Modeling.}
\address{Charles University, Faculty of Mathematics and Physics, Sokolovsk\'a 83, 186 75 Pra\-ha~8, Czech Republic}
\email{kurka.ondrej@seznam.cz, dpokorny@karlin.mff.cuni.cz}

\begin{abstract}
We discuss the following question: For a function $ f $ of two or more variables
which is convex in the directions of coordinate axes, how can its trace
$ g(x) = f(x, x, \dots , x) $ look like? In the two-dimensional case,
we provide some necessary and sufficient conditions,
as well as some examples illustrating that our approach does not seem
to be appropriate for finding a characterization in full generality.
For a concave function $ g $, however, a characterization
in the two-dimensional case is established.
\end{abstract}

\maketitle

\section{Introduction}\label{sec1}

We say that a real function $f:\er^d\to\er$ is {\it separately convex} if it is convex on every line parallel to a coordinate axis.
The notion of separate convexity is investigated due to its relationship with the concept of rank-one convexity
(a function $ f $ on the matrix space $ \er ^{d \times d} $ is said to be \emph{rank-one convex} if it is convex on every line with a rank-one direction).
In order to understand rank-one convexity, one can restrict attention to particular subspaces of $ \er ^{d \times d} $ (e.g. the diagonal or the symmetric matrices).
For example, this approach has been employed by S.~Conti, D.~Faraco, F.~Maggi and S.~M\"uller \cite{CFM, CFMM} in the study of the (still open) question whether
the Hessian of a rank-one convex function is a bounded measure. When the rank-one convexity is considered on the subspace of diagonal matrices, the separate convexity naturally appears.

The study of separately convex functions in the theory of non-linear elasticity goes back to L.~Tartar \cite{T} (see also \cite{KMS,LMM,M}).
In the two-dimensional case, the notion of separate convexity coincides with the notion of bi-convexity used in optimization (see e.g. \cite{GPK}), when considered in $ \er \times \er $.
For aspects of the separate convexity in the theory of martingales, see \cite{AH}, and for recent applications in the studies of removable sets for convex functions, see \cite{PR}.

A real function $g$ on $\er$ is said to be the \emph{trace of $f$} on the diagonal if $g(t)=f(t,t,\dots ,t)$ for every $t$.
In \cite{T}, L.~Tartar asks for the precise class of functions $g$ which can be the trace of a separately convex function on $\er^2$.
The same problem is posed in \cite[Question 11]{KMS}.
He remarks that every $\C^2$ function (or even every function semi-convex in the classical sense) can be a trace, at least locally, and also mentions some unpublished observations by V.~\v Sver\'ak and D.~Preiss.
In \cite{CFM}, S.~Conti, D.~Faraco and F.~Maggi construct a function which can be a trace but its second derivative is not a bounded measure.
They also mention another unpublished result by B.~Kirchheim and A.~Lorent, namely that every $\C^{1,\alpha}$ function can be a trace and that not every $\C^1$ function is a trace.

In fact, it turned out during the preparation of this manuscript that many of the included results were obtained but unpublished more than 10 years ago
by B.~Kirchheim, A.~Lorent and L.~Sz\'ekelyhidi.

In the present paper, we prove some partial results on this topic, mostly in the dimension $2$, but also some observations in the general dimension.
The only case in which we were able to obtain a full characterization is the case of concave functions.
We proved the following results.
\begin{theorem} \label{concave}
Let $ g : \mathbb{R} \rightarrow \mathbb{R} $ be a concave function. Then the following assertions are equivalent:

{\rm (i)} There exists a separately convex function $ f : \mathbb{R}^{2} \rightarrow \mathbb{R} $ such that $ f(u, u) = g(u) $ for each $ u \in \mathbb{R} $.

{\rm (ii)} The function
$$ x \; \mapsto \; \int _{0}^{1} \frac{g(x + t) + g(x - t) - 2g(x)}{t^{2}} \, dt $$
is locally bounded from below.
\end{theorem}
The argument is based on a general extension result (see Theorem~\ref{sufficientthm}) which provides us moreover with a sufficient condition in the framework of semi-convex functions (with a general modulus).
\begin{theorem} \label{semiconvex}
Let a function $ g : \mathbb{R} \rightarrow \mathbb{R} $ be semi-convex with a modulus $ \omega $ such that
$$ \int _{0}^{1} \frac{\omega (t)}{t} \; dt < \infty. $$ 
Then there exists a separately convex function $ f : \mathbb{R}^{2} \rightarrow \mathbb{R} $ such that $ f(u, u) = g(u) $ for each $ u \in \mathbb{R} $.
\end{theorem}
\noindent
These results are special cases of Corollaries~\ref{semiconcave} and \ref{locsemiconvex}.

In fact, the implication $ \mathrm{(i)} \Rightarrow \mathrm{(ii)} $ holds without the assumption for $g$ to be concave, as follows from Proposition~\ref{propdiff} and Remark~\ref{remint}.
A stronger necessary condition is formulated in Theorem~\ref{necess2}.
However, the implication $ \mathrm{(ii)} \Rightarrow \mathrm{(i)} $ does not hold in general and also the property of being a trace cannot be characterized solely in the terms of semi-convexity (cf. Example~\ref{example2} and Example~\ref{example3}).

The paper is structured as follows. After introducing notation and basic facts, we start with some simple results in general dimension. In Section~\ref{sec2}, we discuss the relationship between local and global extendibility, showing that they are essentially the same. Section~\ref{sec3} contains a modest necessary condition for a function to be the trace of a separately convex function on $ \er^{d} $.

A major part of the paper, however, is devoted to the two-dimensional case. In Section~\ref{sec4}, we begin by studying of necessary conditions, namely Proposition~\ref{propdiff}, Theorem~\ref{necess2} and Remark~\ref{remmostgeneral}. Basic tools for our constructions are provided in Section~\ref{sec5}. We formulate two extension Lemmata~\ref{extension} and~\ref{smoothcase2}, the first of which was inspired by the necessary conditions from the previous section. In fact, both of them turned out to be consequences of a general extension result (see Lemma~\ref{extensionasym}). Section~\ref{sec6} contains a technical envelope-like lemma allowing us to apply our extension lemmata to more general functions.

Finally, the main results, Theorems~\ref{concave} and~\ref{semiconvex}, are proven in Section~\ref{sec7}, both of them derived from a general extension theorem (see Theorem~\ref{sufficientthm}). In Section~\ref{sec007}, we present a method that allows us to find a separately convex extension for a very specific but in some sense broad class of traces which do not satisfy the assumptions of Theorem~\ref{sufficientthm}.

We conclude the paper with four examples illustrating the limitations of our methods. These are included in~Section~\ref{sec8}. In particular, we observe that there is a trace of a separately convex function which does not have a one-sided derivative at some point. Note that both main results above provide a separately convex extension only for functions with one-sided derivatives.

The authors are grateful to Bernd Kirchheim for fruitful discussions on the topic, especially on the background of studies of separately convex functions. The authors also thank to Lud\v{e}k Zaj\'i\v{c}ek for valuable discussions on semi-convex functions.

\begin{figure}[t]
\psset{unit=0.333mm}
\begin{pspicture}(-12,-12)(232,112)
\psset{linecolor=gray}
\psset{linewidth=0.6pt}
\psline(-10,0)(230,0)
\psline(0,-10)(0,100)
\psset{linecolor=black}
\psset{linewidth=1pt}
\psline(216,0)(230,0)
\parabola(216,0)(144,72)
\parabola(72,0)(48,24)
\parabola(24,0)(16,8)
\parabola(8,0)(5.33,2.67)
\parabola(2.67,0)(1.77,0.89)
\parabola(0.89,0)(0.59,0.30)
\psline(-10,0)(0.29,0)
\end{pspicture}
\caption{A picture of a trace without one-sided derivative at some point (Example~\ref{example3}).}
\label{nondiff}
\end{figure}
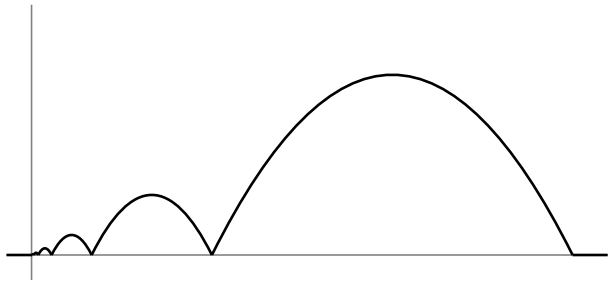

\section*{Notation and basic facts}

In many places, we use the fact that every separately convex function on $\er^d$ is necessarily locally Lipschitz.
In particular, its trace is also locally Lipschitz.

Let $M$ be a convex set and let $f$ be a real function defined on a super set of $M$. We say that the function $f$ is {\it semi-convex with modulus $\omega$ on $M$} if  
$$
f(\alpha x+(1-\alpha)y)\leq \alpha f(x)+(1-\alpha)f(y)+\alpha(1-\alpha)\Vert x-y \Vert \omega(\Vert x-y \Vert )
$$
for every $x,y\in M$ and every $\alpha\in(0,1)$.
A function semi-convex with a linear modulus will be called just semi-convex.
Semi-concave functions are then defined in a similar manner.

For a real function $f$ on $\er$, we define its {\it second order central difference at a point $x$} by $\omega_f(x,t):=f(x+t)+f(x-t)-2f(x)$.
Its modification $\omega_{f}^*(x,t)$ is defined for $ t \geq 0 $ so that $ \omega_{f}^*(x,\cdot) $ is the greatest non-increasing minorant of $ \omega_{f}(x,\cdot) $ on $ [0, \infty) $.

The following observation about traces of separately convex functions appeared in \cite{PR}.

\begin{lemma}\label{sverak}
Let $f:\er^2\to\er$ be a separately convex function. Define $g:\er\to\er$ by $g(t)=f(t,t).$
Then 
$$
 \liminf_{t\to 0+} \frac{g(x+t)+g(x-t)-2g(x)}{t}\geq 0
$$ 
for every $x.$
\end{lemma}

We will also use the following more or less standard notation. 
For a function $ f $ defined on (a subset of) $\er^d$, the symbol $D_{v}^+f(x)$ means the upper one-sided derivative in a direction $v\in S^{d-1}$ of $f$ at a point $x$.

\section{Coincidence of local and global extendibility}\label{sec2}

\begin{lemma} \label{smoothcase}
Let $ d \geq 2 $. Then, for every $ G \in \C^{2}(\mathbb{R}) $, there is a separately convex function $ F : \mathbb{R}^{d} \rightarrow \mathbb{R} $ such that $ F(t, t, \dots , t) = G(t) $ for each $ t \in \mathbb{R} $.
\end{lemma}

\begin{proof}
We will see later how to find an extension $ F_{0} $ for $ d = 2 $ (Lemma~\ref{smoothcase2}). For a general $ d $, it is sufficient to consider the function $ F(x_{1}, x_{2}, \dots , x_{d}) = F_{0}(x_{1}, x_{2}) $.
\end{proof}

\begin{proposition} \label{locglob}
Let $ d \geq 2 $ and let $ g : \mathbb{R} \rightarrow \mathbb{R} $ be a locally Lipschitz function. Assume that, for every bounded interval $ [a, b] $, there is a separately convex function $ f_{[a, b]} : \mathbb{R}^{d} \rightarrow \mathbb{R} $ such that $ f_{[a, b]}(u, u, \dots , u) = g(u) $ for each $ u \in [a, b] $. Then there is a separately convex function $ f : \mathbb{R}^{d} \rightarrow \mathbb{R} $ such that $ f(u, u, \dots , u) = g(u) $ for each $ u \in \mathbb{R} $.
\end{proposition}

We prove a claim first.

\begin{claim}
For every bounded interval $ [a, b] $, there is a separately convex function $ f^{*}_{[a, b]} : \mathbb{R}^{d} \rightarrow \mathbb{R} $ such that $ f^{*}_{[a, b]}(u, u, \dots , u) = g(u) $ for each $ u \in [a, b] $ and $ f^{*}_{[a, b]}(u, u, \dots , u) \leq g(u) $ for each $ u \in \mathbb{R} \setminus [a, b] $.
\end{claim}

\begin{proof}
Let $ G : \mathbb{R} \rightarrow \mathbb{R} $ be a $ \C^{2} $-function such that
$$ G(u) \leq 0, \; u \in \mathbb{R}, \quad G(u) = 0, \; u \in [a, b], $$
$$ G(u) \leq g(u) - f_{[a-1, b+1]}(u, u, \dots , u), \; u \in \mathbb{R} \setminus [a-1, b+1], $$
and let $ F : \mathbb{R}^{d} \rightarrow \mathbb{R} $ be an extension given by Lemma~\ref{smoothcase}. Then the function
$$ f^{*}_{[a, b]} = f_{[a-1, b+1]} + F $$
works.
\end{proof}

\begin{proof}[Proof of Proposition \ref{locglob}]
Let $ G : \mathbb{R} \rightarrow \mathbb{R} $ be a $ \C^{2} $-function such that
$$ G(u) \leq 0, \; u \in \mathbb{R}, \quad G(u) = 0, \; |u| \geq 1, \quad G(0) < - 1, $$
and let $ F : \mathbb{R}^{d} \rightarrow \mathbb{R} $ be an extension given by Lemma~\ref{smoothcase}. We choose an $ \varepsilon \in (0, 1) $ so that
$$ F(x) \leq -1, \quad x \in [-\varepsilon , \varepsilon ]^{d}. $$
For $ k = 0, 1, 2, \dots $, there is a sufficiently large $ \alpha _{k} \geq 0 $ such that the function
$$ f^{(k)}(x) = \max \big\{ f^{*}_{[2^{k}, 2^{k+1}]}(x), f^{*}_{[-2^{k+1}, -2^{k}]}(x) \big\} + \alpha _{k} F \big( \frac{1}{2^{k}} x \big) $$
fulfils
$$ f^{(k)}(x) \leq 0, \quad x \in [-2^{k}\varepsilon , 2^{k}\varepsilon ]^{d}. $$
It follows that every $ x \in \mathbb{R}^{d} $ satisfies $ f^{(k)}(x) \leq 0 $ for all but finitely many $ k $'s. Thus,
$$ f = \sup \Big( \{ f^{*}_{[-1, 1]} \} \cup \big\{ f^{(k)} : k = 0, 1, 2, \dots \big\} \Big) $$
is a well-defined separately convex function. At the same time,
$$ f^{*}_{[-1, 1]}(u, \dots , u) \leq g(u), \; u \in \mathbb{R}, \quad f^{*}_{[-1, 1]}(u, \dots , u) = g(u), \; u \in [-1, 1], $$
$$ f^{(k)}(u, \dots , u) \leq g(u), \; u \in \mathbb{R}, \quad f^{(k)}(u, \dots , u) = g(u), \; u \in [2^{k}, 2^{k+1}] \cup [-2^{k+1}, -2^{k}]. $$
Consequently, we have $ f(u, \dots , u) = g(u) $ for each $ u \in \mathbb{R} $.
\end{proof}

\section{Necessary condition in general dimension}\label{sec3}
Unlike to the two-dimensional case, we were able to obtain only a very weak necessary condition for a function to be the trace of a separately convex function of three or more variables. The procedure we will use in Section~\ref{sec4} can be applied also in dimension three, at least for a concave $g$, but the resulting condition is not a very interesting one, because it is satisfied by every concave function.

\begin{proposition}\label{cord}
Let $d\geq 1$.
Suppose that $f:\er^d\to\er$ is a separately convex function.
Then
$$
D^+_{v}f(x)\geq -D^+_{-v}f(x)
$$ 
for every direction $v\in S^{d-1}$ and $x\in\er^d$.
\end{proposition}

\begin{proof}
The proposition will be proven by induction on $ d $.
For $d=1$ it is sufficient to use the fact that the corresponding inequality holds for every convex function.
(Although it is not needed here, we also note that the validity for $ d = 2 $ follows from Lemma~\ref{sverak}.)
Now, assume that the proposition is valid up to $d-1$ for some $d\geq 2$.
We will prove the validity for $d$.

Aiming for a contradiction suppose that
\begin{equation*}
D^+_{v}f(x)< -D^+_{-v}f(x)
\end{equation*}
for a separately convex function $f:\er^d\to\er$, $x\in\er^d$ and a direction $v\in S^{d-1}$.
Since $f$ is separately convex and by the induction procedure, we can suppose that $v$ is not in the linear hull of any $d-1$ coordinate directions.
This means that we can assume $v=\frac{1}{\sqrt{d}}(1,1,\dots,1)$.
We can also suppose that $x=(0,0,\dots,0)$, $f(x)=0$ and
\begin{equation}\label{derivatives}
D^+_{v}f(x)=-\frac{1}{\sqrt{d}}<\frac{1}{\sqrt{d}}=-D^+_{-v}f(x).
\end{equation}
Moreover, there is no loss in generality in assuming that
\begin{equation}\label{linear}
f(ta)=tf(a)
\end{equation}
for every $a\in\er^d$ and $t\geq 0$.
Indeed, if $f$ satisfies \eqref{derivatives}, so does the function
$$
x\mapsto\limsup_{r\to0+}\frac{1}{r}f(rx),
$$
which additionally satisfies \eqref{linear}.

For $t\in[-1,1]$ put
$$
u_t=(-t,-t,\dots,1),\quad x_t=(-t,-t,\dots,-t)\quad\text{and}\quad y_t=(-t,-t,\dots,t).
$$
Due to the separate convexity, \eqref{derivatives} and \eqref{linear}, we have
$$
\begin{aligned}
f(u_t)\geq f(y_t)+\frac{|u_t-y_t|}{|x_t-y_t|}\cdot (f(y_t)-f(x_t))=\frac{1+t}{2t}f(y_t)-\frac{1-t}{2t}f(x_t)\\
=\frac{1+t}{2}f(y_1)-\frac{1-t}{2}f(x_1)=\frac{1+t}{2}f(y_1)+\frac{1-t}{2}
\end{aligned}
$$
for $ t > 0 $, which gives us (letting $t\to 0+$)
\begin{equation}\label{middlepoint}
f(u_0)\geq \frac{f(y_1)+1}{2}.
\end{equation}

By separate convexity we have for $0<r<s<t<1$
$$
\left(1-\frac{r}{t}\right)f\left(\frac{r}{s}u_s\right)\leq \left(1-\frac{r}{s}\right)f\left(\frac{r}{t}u_t\right)
+\left(\frac{r}{s}-\frac{r}{t}\right)f\left(u_r\right).
$$
Using \eqref{linear} we obtain
$$
\begin{aligned}
\left(1-\frac{r}{t}\right)\frac{r}{s}f\left(u_s\right)\leq& \left(1-\frac{r}{s}\right)\frac{r}{t}f\left(u_t\right)
+\left(\frac{r}{s}-\frac{r}{t}\right)f\left(u_r\right),\\
\frac{r(t-r)}{st}f\left(u_s\right)\leq& \frac{r(s-r)}{st}f\left(u_t\right)+\frac{r(t-s)}{st}f\left(u_r\right),\\
(t-r)f\left(u_s\right)\leq& (s-r)f\left(u_t\right)+(t-s)f\left(u_r\right).
\end{aligned}
$$
This implies that $f$ is convex on the line connecting $u_1$ and $u_0$.
Similar way we obtain that $f$ is convex on the line connecting $u_{-1}$ and $u_0$.
On the other hand
$$
\frac{f(u_1)+f(u_{-1})}{2}=\frac{f(u_1)-1}{2}<\frac{f(y_1)+1}{2}\leq f(u_0)
$$
which tells us that $f$ restricted to $\er^{d-1}\times\{1\}$ is a separately convex function on (a copy of) $\er^{d-1}$ 
such that $D^+_w f(0)< -D^+_{-w}f(0)$ with $w=\frac{1}{\sqrt{d-1}}(1,1,\dots,1)\in S^{d-2}$.
But this is not possible due to the induction procedure, a contradiction.
\end{proof}

\section{Necessary conditions in two dimensions}\label{sec4}
The purpose of this section is to find criteria on a function $ g $ to be the trace on the diagonal of a separately convex function $ f $ of two variables.
We start with some investigation of the behaviour of $ f $ on the diagonals $ x = y $ and $ x = -y $.

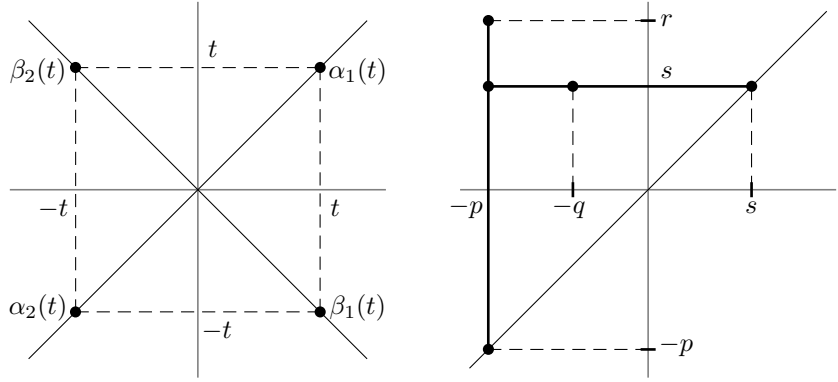
\begin{figure}[t]
\psset{unit=0.25mm}
\begin{pspicture}(-100,-100)(100,100)
\psset{linecolor=gray}
\psset{linewidth=0.6pt}
\psline(-100,0)(100,0)
\psline(0,-100)(0,100)
\psset{linecolor=black}
\psset{linewidth=0.4pt}
\psline(-90,-90)(90,90)
\psline(-90,90)(90,-90)
\psline[linestyle=dashed](65,-65)(65,65)
\psline[linestyle=dashed](65,65)(-65,65)
\psline[linestyle=dashed](-65,65)(-65,-65)
\psline[linestyle=dashed](-65,-65)(65,-65)
\psset{dotscale=1.5}
\psdots(65,65)(65,-65)(-65,-65)(-65,65)
\put(73,-10){\makebox(0,0){$ t $}}
\put(8,75){\makebox(0,0){$ t $}}
\put(-77,-10){\makebox(0,0){$ -t $}}
\put(10,-75){\makebox(0,0){$ -t $}}
\put(85,63){\makebox(0,0){$ \alpha_{1}(t) $}}
\put(-85,-63){\makebox(0,0){$ \alpha_{2}(t) $}}
\put(85,-63){\makebox(0,0){$ \beta_{1}(t) $}}
\put(-85,63){\makebox(0,0){$ \beta_{2}(t) $}}
\end{pspicture}
\hspace{7.5mm}
\psset{unit=0.125mm}
\begin{pspicture}(-200,-200)(200,200)
\psset{linecolor=gray}
\psset{linewidth=0.6pt}
\psline(-200,0)(200,0)
\psline(0,-200)(0,200)
\psset{linecolor=black}
\psset{linewidth=0.4pt}
\psline(-190,-190)(190,190)
\psset{linewidth=1pt}
\psline(-170,-170)(-170,180)
\psline(110,110)(-170,110)
\psline(110,-8)(110,8)
\psline(-80,-8)(-80,8)
\psline(-8,180)(8,180)
\psline(-8,-170)(8,-170)
\psset{linewidth=0.4pt}
\psline[linestyle=dashed](-170,-170)(0,-170)
\psline[linestyle=dashed](110,0)(110,110)
\psline[linestyle=dashed](0,180)(-170,180)
\psline[linestyle=dashed](-80,0)(-80,110)
\psset{dotscale=1.5}
\psdots(110,110)(-80,110)(-170,110)(-170,-170)(-170,180)
\put(-84,-19){\makebox(0,0){$ -q $}}
\put(-194,-19){\makebox(0,0){$ -p $}}
\put(110,-19){\makebox(0,0){$ s $}}
\put(30,-166){\makebox(0,0){$ -p $}}
\put(20,125){\makebox(0,0){$ s $}}
\put(20,180){\makebox(0,0){$ r $}}
\end{pspicture}
\caption{(a) Values of the functions $\alpha_i$, $\beta_i$, $i=1,2$. (b) An illustration concerning of inequalities in \eqref{vypocet1} and \eqref{vypocet2}.}
 \label{fig:alphabeta}
\end{figure}

Consider $f$ a separately convex function on $\er^2$. 
Define functions $\alpha_i$, $\beta_i$, $i=1,2$ on $[0,\infty)$ by (cf. Figure~\ref{fig:alphabeta}(a))
$$
\alpha_1(t):=f(t,t),\;\alpha_2(t):=f(-t,-t),\;\beta_1(t):=f(t,-t)\;\text{and}\;\beta_2(t):=f(-t,t).
$$
Pick $p,q,r,s>0$, $p>q$ and $r>s$. 
Using the separate convexity of $f$ we can obtain (see Figure~\ref{fig:alphabeta}(b))
\begin{equation} \label{vypocet1}
\begin{aligned}
f(-p,s)\geq& f(-q,s)+\frac{p-q}{q+s}\cdot\left(f(-q,s)-f(s,s)\right)\\
=& \frac{q+s+p-q}{q+s}f(-q,s)-\frac{p-q}{q+s}f(s,s)\\
=&\frac{s+p}{q+s}f(-q,s)-\frac{p-q}{q+s}f(s,s)
\end{aligned}
\end{equation}
and similarly
\begin{equation} \label{vypocet2}
\begin{aligned}
f(-p,r)\geq \frac{p+r}{p+s}f(-p,s)-\frac{r-s}{p+s}f(-p,-p).
\end{aligned}
\end{equation}
Putting these two together we then obtain
\begin{equation}
\begin{aligned}
f(-p,r)\geq& \frac{p+r}{p+s}\left(\frac{s+p}{q+s}f(-q,s)-\frac{p-q}{q+s}f(s,s)\right)-\frac{r-s}{p+s}f(-p,-p)\\
=& \frac{p+r}{q+s}\cdot f(-q,s)- \frac{p+r}{p+s}\cdot\frac{p-q}{q+s}\cdot f(s,s)-\frac{r-s}{p+s}f(-p,-p).
\end{aligned}
\end{equation}
This can be rewritten to the more symmetric form
\begin{equation}\label{mostgeneral}
\frac{f(-p,r)}{p+r}\geq \frac{f(-q,s)}{q+s} -\frac{p-q}{(p+s)(q+s)}\cdot f(s,s)-\frac{r-s}{(p+s)(p+r)}\cdot f(-p,-p).
\end{equation}
In the special case of $p=r$ and $q=s$ we then obtain
\begin{equation}\label{lessgeneral}
\frac{f(-r,r)}{r}\geq \frac{f(-s,s)}{s} -\frac{r-s}{s(r+s)}\cdot f(s,s)-\frac{r-s}{r(r+s)}\cdot f(-r,-r).
\end{equation}
Due to the symmetry this implies also
\begin{equation}\label{lessgeneralsym}
\frac{f(r,-r)}{r}\geq \frac{f(s,-s)}{s} -\frac{r-s}{s(r+s)}\cdot f(-s,-s)-\frac{r-s}{r(r+s)}\cdot f(r,r).
\end{equation}
The last two formulas can be rewritten as
\begin{equation}\label{lessgeneral2}
\frac{\beta_2(r)}{r}-\frac{\beta_2(s)}{s}\geq  -\frac{r-s}{s(r+s)}\cdot \alpha_1(s)-\frac{r-s}{r(r+s)}\cdot\alpha_2(r)
\end{equation}
and
\begin{equation}\label{lessgeneralsym2}
\frac{\beta_1(r)}{r}-\frac{\beta_1(s)}{s}\geq  -\frac{r-s}{s(r+s)}\cdot \alpha_2(s)-\frac{r-s}{r(r+s)}\cdot\alpha_1(r).
\end{equation}
Moreover, formulas \eqref{lessgeneral2} and \eqref{lessgeneralsym2} imply
\begin{equation}\label{bezicek}
\frac{\beta(r)}{r}-\frac{\beta(s)}{s}\geq  -\frac{r-s}{r+s}\cdot \left(\frac{\alpha(s)}{s}+\frac{\alpha(r)}{r}\right),
\end{equation}
where we denoted $\alpha=\alpha_1+\alpha_2$ and $\beta=\beta_1+\beta_2$.
It turns out that it is sometimes useful to work with functions $\gamma(t):=\frac{\alpha(t)}{t}$ and $\delta(t):=\frac{\beta(t)}{t}$.
If we rewrite \eqref{bezicek} using $\gamma$ and $\delta$ we obtain
\begin{equation}\label{gammadelta}
\frac{\delta(r)-\delta(s)}{r-s}\geq  -\frac{\gamma(r)+\gamma(s)}{r+s}.
\end{equation}
Suppose that $\delta'(r)$ exists (which is the case for almost every $r$, since $\delta$ is locally Lipschitz on $(0,\infty)$), then
\begin{equation}\label{derivative}
\delta'(r)=\lim_{s\to r-}\frac{\delta(r)-\delta(s)}{r-s}\geq \lim_{s\to r-} -\frac{\gamma(r)+\gamma(s)}{r+s}=-\frac{\gamma(r)}{r}.
\end{equation}
Using the local Lipschitzness of $\delta$ we thus obtain
\begin{equation}\label{interval}
\delta(a)-\delta(b)=\int\limits_{b}^{a}\delta'(t)\;dt\geq -\int\limits_{b}^{a}\frac{\gamma(t)}{t}\;dt
\end{equation}
for $a>b>0$.
This in particular means that
\begin{equation}
\delta(a)-\liminf_{b\to 0+}\delta(b)\geq-\liminf_{b\to 0+}\int\limits_{b}^{a}\frac{\gamma(t)}{t}\;dt.
\end{equation}
Assume from now on that $f(0,0)=0$, then, using Lemma~\ref{sverak}, we obtain that $\liminf_{b\to 0+}\delta(b)\geq 0$.
This gives us
\begin{equation}
\frac{\beta(a)}{a}=\delta(a)\geq-\liminf_{b\to 0+}\int\limits_{b}^{a}\frac{\gamma(t)}{t}\;dt=-\liminf_{b\to 0+}\int\limits_{b}^{a}\frac{\alpha(t)}{t^2}\;dt,
\end{equation}
which can be rewritten to the form
\begin{equation}\label{necess1}
\beta(a)\geq-a\liminf_{b\to 0+}\int\limits_{b}^{a}\frac{\alpha(t)}{t^2}\;dt.
\end{equation}
This gives us
\begin{proposition}\label{propdiff}
Let $g$ be a locally Lipschitz function on $\er$. Suppose that there is a separately convex function $h$ on $\er^2$ such that $g(t)=h(t,t)$ for every $t\in\er$. 
Then 
\begin{equation}\label{odhad}
h(x+u,x-u)+h(x-u,x+u)-2h(x,x)\geq -u\cdot\liminf_{v\to 0+}\int\limits_{v}^{u}\frac{\omega_g(x,t)}{t^2}\;dt,
\end{equation}
for every $x,u\in\er,u>0$.
In particular, for every bounded interval $I\subset\er$, there is a constant $C(g,I)$ such that
$$
-\liminf_{b\to 0+}\int_{b}^{1}\frac{\omega_g(x,t)}{t^2}\;dt<C(g,I)
$$
for every $x\in I$.
\end{proposition}

\begin{proof}
The first part of the proposition follows directly from \eqref{necess1} applied on the separately convex function $f$ defined as $f(y,z)=h(x+y,x+z)-h(x,x)$.
Indeed, in this case the left hand side in \eqref{odhad} is equal to
$$
h(x+u,x-u)+h(x-u,x+u)-2h(x,x)=f(u,-u)+f(-u,u)=\beta(u),
$$
whereas the argument of the integral inside the \uv{$\liminf$} on the right hand side is equal to
$$
\begin{aligned}
\frac{\omega_g(x,t)}{t^2}=&\frac{g(x+t)+g(x-t)-2g(x)}{t^2}\\
=& \frac{h(x+t,x+t)+h(x-t,x-t)-2h(x,x)}{t^2}\\
=& \frac{f(t,t)+f(-t,-t)}{t^2}
=\frac{\alpha(t)}{t^2}.
\end{aligned}
$$

The second part of the proposition follows immediately from \eqref{odhad} and from the local Lipschitzness of $h$.
\end{proof}

Now, return to formula \eqref{bezicek} and consider the special case of $\alpha(s)=\alpha(r)$ which gives us
\begin{equation}\label{constant}
\begin{aligned}
\frac{\beta(r)}{r}-\frac{\beta(s)}{s}\geq&  -\frac{r-s}{r+s}\cdot \left(\frac{\alpha(s)}{s}+\frac{\alpha(s)}{r}\right)
= -\frac{r-s}{r+s}\cdot \frac{r+s}{rs}\cdot\alpha(s)\\
=& -\frac{r-s}{rs}\cdot\alpha(s)
= -\left(\frac{\alpha(s)}{s}-\frac{\alpha(s)}{r}\right)
= -\int\limits_{s}^{r}\frac{\alpha(s)}{t^2}\;dt.
\end{aligned}
\end{equation} 

Pick some $0<s_1<r_1<s_2<r_2<\cdots<s_k<r_k<w$ with $\alpha(s_i)=\alpha(r_i)$, $i=1,\dots,k$ and consider the function $\tilde{\alpha}$ on $[0,\infty)$ defined as
$$
\tilde{\alpha}(t):=
\begin{cases}
\alpha(s_i) & \text{if}\quad t\in[s_i,r_i],\\
\alpha(t)  & \text{otherwise}.
\end{cases}
$$
By \eqref{constant} and \eqref{interval} we obtain that
$$
\frac{\beta(r_i)}{r_i}-\frac{\beta(s_i)}{s_i}\geq -\int\limits_{s_i}^{r_i}\frac{\alpha(s_i)}{t^2}\;dt=-\int\limits_{s_i}^{r_i}\frac{\tilde\alpha(t)}{t^2}\;dt,
$$
$$
\frac{\beta(s_{i+1})}{s_{i+1}}-\frac{\beta(r_i)}{r_i}\geq -\int\limits_{r_i}^{s_{i+1}}\frac{\alpha(t)}{t^2}\;dt=-\int\limits_{r_i}^{s_{i+1}}\frac{\tilde\alpha(t)}{t^2}\;dt
$$
and
$$
\frac{\beta(w)}{w}-\frac{\beta(r_k)}{r_k}\geq -\int\limits_{r_k}^{w}\frac{\alpha(t)}{t^2}\;dt=-\int\limits_{r_k}^{w}\frac{\tilde\alpha(t)}{t^2}\;dt.
$$
Putting these inequalities together we then obtain
\begin{equation}\label{skonstantou}
\frac{\beta(w)}{w}-\frac{\beta(s_1)}{s_1}\geq -\int\limits_{s_1}^{w}\frac{\tilde\alpha(t)}{t^2}\;dt.
\end{equation}

The formula \eqref{skonstantou} leads to the following
\begin{theorem}\label{necess2}
Let $g$ be a locally Lipschitz function on $\er$. Suppose that there is a separately convex function $h$ on $\er^2$ such that $g(t)=h(t,t)$ for every $t\in\er$. 
Then for every bounded interval $I\subset\er$, there is a constant $C^{*}(g,I)$ such that
$$
-\int\limits_{0}^{1}\frac{\omega^{*}_g(x,t)}{t^2}\;dt<C^{*}(g,I)
$$
for every $x\in I$.
\end{theorem}

\begin{proof}
Fix a bounded interval $I=[a,b]$, pick $x\in I$ and let $p\in[0,1]$ be the last point of a global minimum of $\omega_g(x,\cdot)$ on $[0,1]$.

Let $I_n=[a_n,b_n]$, $n\in A$ be the (at most) countable system of non-degenerate maximal intervals in $[0,p]$ with the property that the function $\omega^{*}_g(x,t)$ 
(as a function of $t$) is constant on every $I_n$.
Then $\omega_g(x,a_n)=\omega_g(x,b_n)$ and $\omega_g(x,t)\geq \omega^{*}_g(x,t)$ for $t\in I_n$ for every $n\in A$.
We will consider only the (most difficult) case, when $\liminf_{n\to\infty} a_n=0$, the other cases can be resolved similarly.
So suppose that there are $n_1,n_2,\dots$ such that $c_i:=a_{n_i}\searrow 0$.

Define real functions $\kappa_{n}$ on $[0,p]$, $n\in\en$ by
$$
\kappa_n(t):=
\begin{cases}
\omega_g(x,a_k) & \text{if}\quad t\in I_k \quad\text{and}\quad k\leq n,\\
\omega_g(x,t)  & \text{otherwise}.
\end{cases}
$$
Then $\kappa_{n}\searrow\omega^{*}_g(x,\cdot )$ as $n\to\infty$.
Define
$$
\delta(t):=\frac{h(x+t,x-t)+h(x-t,x+t)-2h(x)}{t}.
$$
Then using \eqref{skonstantou} together with Lebesgue's Monotone Convergence Theorem
we obtain that 
\begin{equation}\label{odhadmon}
\delta(p)-\delta(c_i)\geq -\sup_n\int\limits_{c_i}^{p}\frac{\kappa_{n}(t)}{t^2}\;dt=-\int\limits_{c_i}^{p}\lim_{n\to\infty}\frac{\kappa_{n}(t)}{t^2}\;dt=-\int\limits_{c_i}^{p}\frac{\omega^{*}_g(x,t)}{t^2}\;dt.
\end{equation}
Since the value $\omega^*_g(x,t)$ is always non-positive we obtain similarly as in the case of formula \eqref{necess1} that
$$
2\sqrt{2}L\geq \delta(p)\geq-\liminf_{i\to\infty}\int_{c_i}^{p}\frac{\omega^{*}_g(x,t)}{t^2}\;dt=-\int_{0}^{p}\frac{\omega^*_g(x,t)}{t^2}\;dt,
$$ 
where $L$ is the Lipschitz constant of $h$ on $[a-1,b+1]^2$.

It remains to estimate the integral from $p$ to $1$.
Let $K$ be the Lipschitz constant of $g$ on $[a-1,b+1]$.
Then $\omega_g(x,p)\geq -2p\cdot K$ and hence
$$
-\int\limits_{p}^1\frac{\omega^{*}_g(x,t)}{t^2}\;dt\leq 2p\cdot K\int\limits_{p}^1\frac{1}{t^2}\;dt=2p\cdot K\left(\frac{1}{p}-1\right)=2K(1-p)\leq 2K
$$
so we are done since $K$ depends only on $I$ and $g$.
\end{proof}

\begin{remark}\label{remint}
From Theorem~\ref{necess2} one can easily see that under the assumptions of Proposition~\ref{propdiff} the integral
$$
\int_{0}^{1}\frac{\omega_g(x,t)}{t^2}\;dt
$$
exists.
\end{remark}

\begin{remark}\label{remmostgeneral}
If we decided to push this direction of finding necessary conditions for a real function to be the trace of a separately convex function to its limits, we might formulate it as follows:

{\it Let $g$ be a locally Lipschitz function on $\er$. Suppose that there is a separately convex function $h$ on $\er^2$ such that $g(t)=h(t,t)$ for every $t\in\er$. 
Then for every bounded interval $I\subset\er$, there is a constant $C(g,I)$ such that }
\begin{equation}\label{mostgenfor}
\begin{aligned}
-\sum_{i=1}^{m-1}\Biggl[ &\frac{p_i-p_{i+1}}{(p_i+r_{i+1})(p_{i+1}+r_{i+1})}\cdot \left(g(x+r_{i+1})-g(x)\right)\\
& +\frac{r_i-r_{i+1}}{(p_i+r_{i+1})(p_i+r_i)}\cdot \left(g(x-p_i)-g(x)\right)\Biggr]<C(g,I)
\end{aligned}
\end{equation}
{\it for every $x\in I$, $m\in\en$ and every pair of decreasing sequences $\{p_i\}_{i=1}^m$ and $\{r_i\}_{i=1}^m$ from $(0,1]$.}
Indeed, it is sufficient (for a fixed $x$) to consider formula \eqref{mostgeneral} for $f$ defined by $f(u,v):=h(u+x,v+x)-h(x,x)$ with the choice $p=p_i$, $q=p_{i+1}$, $r=r_i$, $s=r_{i+1}$ which gives us
$$
\begin{aligned}
\frac{f(-p_i,r_i)}{p_i+r_i}-\frac{f(-p_{i+1},r_{i+1})}{p_{i+1}+r_{i+1}} \geq &-\frac{p_i-p_{i+1}}{(p_i+r_{i+1})(p_{i+1}+r_{i+1})}\cdot f(r_{i+1},r_{i+1})\\
&-\frac{r_i-r_{i+1}}{(p_i+r_{i+1})(p_i+r_i)}\cdot f(-p_i,-p_i).
\end{aligned}
$$
Now, it remains to take a sum over $i$ to obtain 
$$
\begin{aligned}
\frac{f(-p_1,r_1)}{p_1+r_1}-\frac{f(-p_{m},r_{m})}{p_{m}+r_{m}} \geq -\sum_{i=1}^{m-1}\Biggl[ &\frac{p_i-p_{i+1}}{(p_i+r_{i+1})(p_{i+1}+r_{i+1})}\cdot f(r_{i+1},r_{i+1})\\
&+\frac{r_i-r_{i+1}}{(p_i+r_{i+1})(p_i+r_i)}\cdot f(-p_i,-p_i)\Biggr]
\end{aligned}
$$
and use the local Lipschitzness of $h$. 
There doesn't seem to be, however, a direct geometric interpretation of this condition similar to Theorem~\ref{necess2}.
It might be worth mentioning though that the formula \eqref{mostgeneral} also has an integral-like form
\begin{equation}\label{mostgeneralint}
\frac{f(-p,r)}{p+r}\geq \frac{f(-q,s)}{q+s} -\int\limits_{q+s}^{p+s}\frac{f(s,s)}{t^2}\;dt -\int\limits_{p+s}^{p+r}\frac{f(-p,-p)}{t^2}\;dt.
\end{equation}
\end{remark}

\begin{remark}
It is also worth noting that (due to the symmetry) the formula \eqref{gammadelta} holds when replacing $\delta$ with $\gamma$ and vice versa, this gives us
\begin{equation}\label{gammadeltasym}
\frac{\gamma(r)-\gamma(s)}{r-s}\geq  -\frac{\delta(r)+\delta(s)}{r+s}.
\end{equation}
Considering the limit $s\to r-$, the formula \eqref{gammadeltasym} implies that
\begin{equation}\label{alphabetasym}
\beta(r)\geq  \alpha(r)-r\alpha'(r),
\end{equation}
provided $\alpha'(r)$ exists.
This inequality sometimes gives a better estimate than \eqref{necess1}. 
For instance for the function $g(t)=-|t|^3$ the formula \eqref{necess1} gives 
$$
f(t,-t)+f(-t,t)\geq t\int_{0}^{t}2s\;ds =t^3, \quad t>0,
$$
however, the formula \eqref{alphabetasym} gives
$$
f(t,-t)+f(-t,t)\geq g(t)+g(-t)-t(g'(t)-g'(-t))=-2t^3+6t^3=4t^3, \quad t>0.
$$
This observation is also implicitly included in the proof of Lemma~\ref{smoothcase2}, namely in the definition of the function $\eta$.
\end{remark}

\section{Extension lemmata}\label{sec5}

We start this section with a general extension lemma which we then apply in two particular situations (Lemmata \ref{extension} and \ref{smoothcase2}). The lemma is based on a simple geometric idea. If the values on the diagonals $ x = \pm y $ are given, we consider the function which is affine on every line segment which is parallel to a coordinate axis with the endpoints belonging to these diagonals (i.e., the dashed line segments in Figure~\ref{fig:alphabeta}(a)).

In fact, the first extension result we obtained was Lemma \ref{extension} in which the values on the diagonal $ x = -y $ are extracted from the formula (\ref{necess1}). Note that the second extension result Lemma \ref{smoothcase2} assumes the symmetric formula (\ref{alphabetasym}) at the same time.

\begin{lemma} \label{extensionasym}
Let $ \alpha _{1}, \alpha _{2}, \beta _{1}, \beta _{2} : \mathbb{R} \rightarrow \mathbb{R} $ be locally Lipschitz even functions such that, for all choices $ i, j \in \{ 1, 2 \} $,
$$ \lim _{x \rightarrow 0} \frac{\alpha _{i}(x)}{x} = 0, \quad \lim _{x \rightarrow 0} \frac{\beta _{j}(x)}{x} = 0 $$
and the functions
$$ \alpha _{i}'(x) + \frac{1}{x} \big( \beta _{j}(x) - \alpha _{i}(x) \big) , \quad \beta _{j}'(x) + \frac{1}{x} \big( \alpha _{i}(x) - \beta _{j}(x) \big) , $$
are non-decreasing on their domains for $ x > 0 $.

Then the function $ F : \mathbb{R}^{2} \rightarrow \mathbb{R} $ given by
$$ F(x,y) = \left\{\begin{array}{ll}
\frac{1}{2x} [(x+y)\alpha _{1}(x) + (x-y)\beta _{1}(x)], & \quad x \geq |y|, \, x \neq 0, \\
\frac{1}{2x} [(x+y)\alpha _{2}(x) + (x-y)\beta _{2}(x)], & \quad x \leq -|y|, \, x \neq 0, \\
\frac{1}{2y} [(y+x)\alpha _{1}(y) + (y-x)\beta _{2}(y)], & \quad y > |x|, \\
\frac{1}{2y} [(y+x)\alpha _{2}(y) + (y-x)\beta _{1}(y)], & \quad y < -|x|, \\
0, & \quad x = y = 0,
\end{array} \right. $$
is separately convex.
\end{lemma}

\begin{claim} \label{extensionasymclaim}
Let $ E $ be a subset of $ (0, \infty ) $ for which $ \lambda ((0, \infty ) \setminus E) = 0 $ and let $ \gamma : E \rightarrow \mathbb{R} $ be non-decreasing on $ E $. Then
$$ \gamma (q) - \gamma (p) + y \bigg[ \frac{\gamma (q)}{q} - \frac{\gamma (p)}{p} + \int _{p}^{q} \frac{\gamma (t)}{t^{2}} \, dt \bigg] \geq 0 $$
whenever $ |y| \leq p \leq q $ and $ p, q \in E $.
\end{claim}

\begin{proof}
It is sufficient to prove the inequality for $ y = \pm p $. Since
$$ \int _{p}^{q} \frac{\gamma (t)}{t^{2}} \, dt \geq \int _{p}^{q} \frac{\gamma (p)}{t^{2}} \, dt = \gamma (p) \cdot \Big( \frac{1}{p} - \frac{1}{q} \Big) , $$
we can write
\begin{align*}
\gamma (q) - \gamma (p) + p & \bigg[ \frac{\gamma (q)}{q} - \frac{\gamma (p)}{p} + \int _{p}^{q} \frac{\gamma (t)}{t^{2}} \, dt \bigg] \\
 & \geq p \bigg[ \frac{\gamma (q)}{q} - \frac{\gamma (p)}{p} + \frac{\gamma (p)}{p} - \frac{\gamma (p)}{q} \bigg] = \frac{p}{q} \big( \gamma (q) - \gamma (p) \big) \geq 0.
\end{align*}
Since
$$ \int _{p}^{q} \frac{\gamma (t)}{t^{2}} \, dt \leq \int _{p}^{q} \frac{\gamma (q)}{t^{2}} \, dt = \gamma (q) \cdot \Big( \frac{1}{p} - \frac{1}{q} \Big) , $$
we can write
\begin{align*}
\gamma (q) - \gamma (p) - p & \bigg[ \frac{\gamma (q)}{q} - \frac{\gamma (p)}{p} + \int _{p}^{q} \frac{\gamma (t)}{t^{2}} \, dt \bigg] \\
 & \geq \gamma (q) - \gamma (p) - \frac{p}{q} \gamma (q) + \gamma (p) - \gamma (q) + \frac{p}{q} \gamma (q) = 0.
\end{align*}
\end{proof}

\begin{proof}[Proof of Lemma \ref{extensionasym}]
(I) We prove first that $ F $ is separately convex on the quarters from its definition. We consider the set $ \{ (x,y) : x \geq |y| \} $ only, since the proof for other quarters is essentially the same. For the simplicity, we write $ \alpha , \beta $ instead of $ \alpha _{1}, \beta _{1} $. Notice that $ F $ is affine in the direction $ y $, so we just need to show that $ F $ is convex in the direction $ x $.

Let $ y $ be fixed. For almost every $ x > |y| $, we have
\begin{equation} \label{comp1} F_{x}(x, y) = \frac{1}{2} \big( \alpha '(x) + \beta '(x) \big) + \frac{y}{2x} \big( \alpha '(x) - \beta '(x) \big) - \frac{y}{2x^{2}} \big( \alpha (x) - \beta (x) \big) . \end{equation}
We want to show that this partial derivative is non-decreasing on its domain, that is to prove that
$$ \alpha '(p) + \beta '(p) + \frac{y}{p} \big( \alpha '(p) - \beta '(p) \big) - \frac{y}{p^{2}} \big( \alpha (p) - \beta (p) \big) $$
$$ \leq \alpha '(q) + \beta '(q) + \frac{y}{q} \big( \alpha '(q) - \beta '(q) \big) - \frac{y}{q^{2}} \big (\alpha (q) - \beta (q) \big) $$
whenever $ |y| < p \leq q $ and the derivatives exist.

Define
$$ \gamma (x) = \alpha '(x) + \frac{1}{x} \big( \beta (x) - \alpha (x) \big) , $$
$$ \delta (x) = \beta '(x) + \frac{1}{x} \big( \alpha (x) - \beta (x) \big) , $$
$$ \varrho (x) = \frac{1}{x^{2}} \big( \alpha (x) - \beta (x) \big) . $$
We can compute
$$ \varrho '(x) = \frac{1}{x^{2}} \big( \alpha '(x) - \beta '(x) \big) - \frac{2}{x^{3}} \big( \alpha (x) - \beta (x) \big) = \frac{1}{x^{2}} \big( \gamma (x) - \delta (x) \big) . $$
By Claim \ref{extensionasymclaim},
$$ \gamma (q) - \gamma (p) + y \bigg[ \frac{\gamma (q)}{q} - \frac{\gamma (p)}{p} + \int _{p}^{q} \frac{\gamma (t)}{t^{2}} \, dt \bigg] \geq 0, $$
$$ \delta (q) - \delta (p) - y \bigg[ \frac{\delta (q)}{q} - \frac{\delta (p)}{p} + \int _{p}^{q} \frac{\delta (t)}{t^{2}} \, dt \bigg] \geq 0. $$
Summing up these inequalities and using the formula for $ \varrho ' $, we obtain
$$ \gamma (q) + \delta (q) - \gamma (p) - \delta (p) + y \bigg[ \frac{\gamma (q)}{q} - \frac{\delta (q)}{q} - \frac{\gamma (p)}{p} + \frac{\delta (p)}{p} + \int _{p}^{q} \varrho '(t) \, dt \bigg] \geq 0. $$
That is,
$$ \alpha '(q) + \beta '(q) - \alpha '(p) - \beta '(p) + y \bigg[ \frac{1}{q} \big( \alpha '(q) - \beta '(q) \big) - \frac{2}{q^{2}} \big( \alpha (q) - \beta (q) \big) $$
$$ - \frac{1}{p} \big( \alpha '(p) - \beta '(p) \big) + \frac{2}{p^{2}} \big( \alpha (p) - \beta (p) \big) + \frac{1}{q^{2}} \big( \alpha (q) - \beta (q) \big) - \frac{1}{p^{2}} \big( \alpha (p) - \beta (p) \big) \bigg] \geq 0, $$
which leads quickly to the desired inequality.

(II) It remains to show that the separate convexity is not disrupted on the diagonals $ x = \pm y $. Due to the symmetry, we show only that
$$ F_{x-}(a, a) \leq F_{x+}(a, a) \quad \textrm{for a.e. $ a > 0 $.} $$
Using (\ref{comp1}), we rewrite this requirement in the form
$$ \frac{1}{2a} [\alpha _{1}(a) - \beta _{2}(a)] \leq \alpha _{1}'(a) - \frac{1}{2a} \big( \alpha _{1}(a) - \beta _{1}(a) \big) \quad \textrm{for a.e. $ a > 0 $.} $$
It is sufficient to realize that
$$ \alpha _{1}'(x) + \frac{1}{x} \big( \beta _{1}(x) - \alpha _{1}(x) \big) \geq 0 \quad \textrm{and} \quad \alpha _{1}'(x) + \frac{1}{x} \big( \beta _{2}(x) - \alpha _{1}(x) \big) \geq 0 $$
for $ x > 0 $ on the domain of these functions. These functions are assumed to be non-decreasing for $ x > 0 $. They have the limits from the right at $ 0 $, considered with respect to the domain of $ \alpha _{1}' $, and we check that these limits are equal to $ 0 $. Since $ \lim _{x \rightarrow 0+} \frac{1}{x} (\beta _{j}(x) - \alpha _{i}(x)) = 0 $, the function $ \alpha _{1}' $ itself has the limit. The limit must be $ 0 $, as $ \lim _{x \rightarrow 0+} \frac{1}{x} \alpha _{1}(x) = 0 $.
\end{proof}

\begin{lemma} \label{extension}
Let $ \alpha : \mathbb{R} \rightarrow \mathbb{R} $ be a locally Lipschitz even function such that
\begin{itemize}
\item $ \lim _{x \rightarrow 0} \alpha (x) /x = 0 $,
\item the integral $ \int _{0}^{1} \alpha (x) \, dx /x^{2} $ is convergent,
\item the function $ \alpha '(x)/x $ is non-decreasing on its domain for $ x > 0 $.
\end{itemize}
Let $ \beta : \mathbb{R} \rightarrow \mathbb{R} $ be given by
$$ \beta (x) = -x \int _{0}^{x} \frac{\alpha (t)}{t^{2}} \, dt $$
and $ F : \mathbb{R}^{2} \rightarrow \mathbb{R} $ be given by
$$ F(x,y) = \left\{\begin{array}{ll}
\frac{1}{2x} [(x+y)\alpha (x) + (x-y)\beta (x)], & \quad |x| \geq |y|, \, x \neq 0, \\
\frac{1}{2y} [(y+x)\alpha (y) + (y-x)\beta (y)], & \quad |x| < |y|, \\
0, & \quad x = y = 0.
\end{array} \right. $$
Then the function $ F $ is separately convex.
\end{lemma}

\begin{proof}
Due to Lemma \ref{extensionasym}, it is sufficient to show that the functions
$$ \gamma (x) = \alpha '(x) + \frac{1}{x} \big( \beta (x) - \alpha (x) \big) , \quad \delta (x) = \beta '(x) + \frac{1}{x} \big( \alpha (x) - \beta (x) \big) , $$
are non-decreasing on their domains for $ x > 0 $. We can compute
\begin{equation} \label{comp3} \beta '(x) = - \int _{0}^{x} \frac{\alpha (t)}{t^{2}} \, dt - x \cdot \frac{\alpha (x)}{x^{2}} = \frac{1}{x} \big( \beta (x)-\alpha (x) \big) , \end{equation}
\begin{equation} \label{comp4} \beta ''(x) = - \frac{\alpha (x)}{x^{2}} - \frac{1}{x^{2}} \big( \alpha '(x) x - \alpha (x) \big) = - \frac{1}{x} \alpha '(x). \end{equation}
It follows that
$$ \gamma (x) = \alpha '(x) + \beta '(x) \quad \textrm{and} \quad \delta (x) = 0. $$
Now, let $ 0 < p < q $ be elements of the domain of $ \gamma $. Then
\begin{eqnarray*}
\gamma (q) - \gamma (p) & = & \alpha '(q) - \alpha '(p) + \int _{p}^{q} \beta ''(x) \, dx = \alpha '(q) - \alpha '(p) - \int _{p}^{q} \frac{\alpha '(x)}{x} \, dx \\
 & \geq & \alpha '(q) - \alpha '(p) - \int _{p}^{q} \frac{\alpha '(q)}{q} \, dx = \alpha '(q) - \alpha '(p) - (q - p) \cdot \frac{\alpha '(q)}{q} \\
 & = & p \cdot \Big( \frac{\alpha '(q)}{q} - \frac{\alpha '(p)}{p} \Big) \geq 0.
\end{eqnarray*}
\end{proof}

\begin{remark} \label{extensionbound}
It might be useful to have an upper bound for the function $ F $ from Lemma~\ref{extension}. It is possible to prove that, if there is a $ p > 0 $ such that $ \alpha '(x)/x $ is constant on $ [p, \infty ) $, then
$$ F(x,y) \leq \alpha \Big( \frac{1}{2}(x+y) \Big) + \beta \Big( \frac{1}{2}(x-y) \Big) + C, \quad x, y \in \mathbb{R}, $$
where $ C = \frac{1}{2} p \alpha '(p) - \alpha (p) $.
\end{remark}

\begin{lemma} \label{smoothcase2}
For every $ G \in \C^{2}(\mathbb{R}) $, there is a separately convex function $ F : \mathbb{R}^{2} \rightarrow \mathbb{R} $ such that $ F(t, t) = G(t) $ for each $ t \in \mathbb{R} $.
\end{lemma}

\begin{proof}
Without loss of generality, we suppose that $ G(0) = 0 $ and $ G'(0) = 0 $. Let $ M > 0 $ be such that $ |G''(x)| \leq M $ for $ |x| \leq 1 $. For every $ x \in \mathbb{R} $, we define
$$ \alpha _{1}(x) = G(|x|), \quad \quad \alpha _{2}(x) = G(-|x|), $$
so we have
$$ |\alpha ''_{i}(x)| \leq M, \quad |\alpha '_{i}(x)| \leq Mx, \quad |\alpha _{i}(x)| \leq \frac{1}{2} Mx^{2}, \quad i = 1, 2, \; 0 < x \leq 1. $$
For every $ x > 0 $, we furthermore define
$$ \theta _{i}(x) = \frac{1}{x} \alpha _{i}(x), \quad i = 1, 2, \quad \quad \theta (x) = \int _{0}^{x} \min \{ \theta '_{1}(y), \theta '_{2}(y) \} \, dy, $$
$$ \eta (x) = - \theta (x) + \sup \Big\{ -y \alpha ''_{i}(y) + \alpha '_{i}(y) : 0 < y \leq x, \, i = 1, 2 \Big\} . $$
Since
\begin{equation} \label{comp15} \theta '_{i}(x) = \frac{1}{x} \alpha '_{i}(x) - \frac{1}{x^{2}} \alpha _{i}(x), \end{equation}
we have
$$ |\theta '_{i}(x)| \leq \frac{3}{2} M, \quad |\theta (x)| \leq \frac{3}{2} Mx, \quad |\eta (x)| \leq \frac{7}{2} Mx, \quad i = 1, 2, \; 0 < x \leq 1. $$
So, for $ x \in \mathbb{R} $, we can define
$$ \beta (x) = \beta _{1}(x) = \beta _{2}(x) = |x| \int _{0}^{|x|} \frac{\eta (y)}{y} \, dy. $$
Since
\begin{equation} \label{comp16} \beta '(x) = \int _{0}^{x} \frac{\eta (y)}{y} \, dy + x \cdot \frac{\eta (x)}{x} = \frac{1}{x}\beta (x) + \eta (x), \quad x > 0, \end{equation}
$ \beta $ is locally Lipschitz.

To show that Lemma \ref{extensionasym} can be applied, we have to check that, given $ i \in \{ 1, 2 \} $, the functions
$$ \gamma (x) = \alpha _{i}'(x) + \frac{1}{x} \big( \beta (x) - \alpha _{i}(x) \big) , \quad \delta (x) = \beta '(x) + \frac{1}{x} \big( \alpha _{i}(x) - \beta (x) \big) , $$
are non-decreasing for $ x > 0 $. We have
$$ \gamma (x) = \alpha _{i}'(x) + \int _{0}^{x} \frac{\eta (y)}{y} \, dy - \theta _{i}(x), $$
thus we obtain from (\ref{comp15}) that
$$ \gamma '(x) = \alpha _{i}''(x) + \frac{\eta (x)}{x} - \frac{1}{x} \alpha '_{i}(x) + \frac{1}{x^{2}} \alpha _{i}(x). $$
By the definitions of $ \eta $ and $ \theta $,
$$ \eta (x) \geq - \theta (x) - x \alpha ''_{i}(x) + \alpha '_{i}(x) \geq - \theta _{i}(x) - x \alpha ''_{i}(x) + \alpha '_{i}(x), $$
and it follows that $ \gamma '(x) \geq 0 $. Hence, $ \gamma $ is non-decreasing for $ x > 0 $ indeed. Further, using (\ref{comp16}), we get for $ x > 0 $ that
$$ \delta (x) = \eta (x) + \theta _{i}(x) = \theta _{i}(x) - \theta (x) + \sup \Big\{ -y \alpha ''_{j}(y) + \alpha '_{j}(y) : 0 < y \leq x, \, j = 1, 2 \Big\} . $$
Therefore, $ \delta $ is the sum of two functions which are non-decreasing.
\end{proof}

\section{One more lemma}\label{sec6}

Lemma \ref{extension} allows us to find an extension for a substantially restricted class of functions. In the following lemma, we minorize a more general function by a function which meets the assumptions of Lemma \ref{extension} and has the same value at a given point.

\begin{lemma} \label{envelope}
Let $ \varphi : [0, r] \rightarrow (-\infty , 0] $ be a non-increasing function such that $ \varphi (x)/x \rightarrow 0 $ as $ x \searrow 0 $. Let $ \alpha : \mathbb{R} \rightarrow (-\infty , 0] $ be given by
$$ \alpha (x) = \sup \Big\{ ax^{2} + c : a, c \leq 0, ay^{2} + c \leq \varphi (y) \textrm{ for } 0 \leq y \leq r \Big\} , \quad x \in \mathbb{R}. $$
Then
\begin{itemize}
\item $ \alpha $ is Lipschitz and $ \alpha (x)/x \rightarrow 0 $ as $ x \rightarrow 0 $,
\item $ \alpha '(x)/x $ is non-decreasing on its domain for $ x > 0 $,
\item $ \int _{0}^{\infty } \frac{\alpha (x)}{x^{2}} \, dx \geq 2 \frac{\varphi (r)}{r} + 3 \int _{0}^{r} \frac{\varphi (x)}{x^{2}} \, dx $.
\end{itemize}
\end{lemma}

The proof of the lemma is provided in several steps. Without loss of generality, we assume that $ \varphi $ is lower semi-continuous. We define
$$ \psi (0) = 0, \quad \psi(x) = -\frac{\varphi (x)}{x}, \quad 0 < x \leq r. $$
For each $ x > 0 $, we choose $ a_{x} \leq 0 $ and $ c_{x} \leq 0 $ so that
$$ a_{x}x^{2} + c_{x} = \alpha (x) \quad \textrm{and} \quad a_{x}y^{2} + c_{x} \leq \varphi (y) \textrm{ for } 0 \leq y \leq r. $$

\begin{claim}
Such $ a_{x} $ and $ c_{x} $ can be chosen.
\end{claim}

\begin{proof}
Let us consider sequences $ a_{x}^{n} \leq 0 $ and $ c_{x}^{n} \leq 0 $ such that
$$ a_{x}^{n}y^{2} + c_{x}^{n} \leq \varphi (y) \textrm{ for } 0 \leq y \leq r $$
and
$$ a_{x}^{n}x^{2} + c_{x}^{n} \rightarrow \alpha (x). $$
It is sufficient to show that the sequences are bounded, as then a subsequence of $ (a_{x}^{n}, c_{x}^{n}) $ converges. Let
$$ \kappa = \min _{n \in \mathbb{N}} \{ a_{x}^{n}x^{2} + c_{x}^{n} \} . $$
We can write
$$ \kappa \leq a_{x}^{n}x^{2} + c_{x}^{n} \leq a_{x}^{n}x^{2}, \quad \kappa \leq a_{x}^{n}x^{2} + c_{x}^{n} \leq c_{x}^{n}, $$
and so $ 0 \geq a_{x}^{n} \geq \kappa /x^{2}, \, 0 \geq c_{x}^{n} \geq \kappa $.
\end{proof}

\begin{claim} \label{envcl1}
The function $ x \mapsto a_{x} $ is non-decreasing.
\end{claim}

\begin{proof}
Let $ 0 < x < y $. We have
$$ a_{x}x^{2} + c_{x} = \alpha (x) \geq a_{y}x^{2} + c_{y}, $$
$$ a_{y}y^{2} + c_{y} = \alpha (y) \geq a_{x}y^{2} + c_{x}. $$
Summing up these inequalities,
$$ a_{x}x^{2} + a_{y}y^{2} + c_{x} + c_{y} \geq a_{y}x^{2} + a_{x}y^{2} + c_{y} + c_{x}, $$
i.e.,
$$ a_{y}(y^{2}-x^{2}) \geq a_{x}(y^{2}-x^{2}). $$
Consequently, $ a_{y} \geq a_{x} $. 
\end{proof}

\begin{claim} \label{envcl2}
The function $ \alpha '(x)/x $ is non-decreasing on its domain for $ x > 0 $.
\end{claim}

\begin{proof}
By Claim \ref{envcl1}, it is sufficient to realize that
$$ \frac{\alpha '(x)}{x} = 2a_{x} \quad \textrm{if $ \alpha '(x) $ exists.} $$
Since the function $ y \mapsto a_{x}y^{2} + c_{x} $ is a minorant of $ \alpha $ and has the same value at $ x $, it has also the same derivative. That is, $ \alpha '(x) = 2a_{x}x $.
\end{proof}

\begin{claim} \label{envcl3}
$ \alpha (x)/x \rightarrow 0 $ as $ x \rightarrow 0 $.
\end{claim}

\begin{proof}
Let $ \varepsilon > 0 $. We want to find a $ \delta > 0 $ such that
$$ \alpha (x) \geq -\varepsilon x, \quad 0 < x \leq \delta . $$
Since $ \varphi (x)/x \rightarrow 0 $ as $ x \searrow 0 $, there is a $ \delta _{0} > 0 $ such that
$$ \varphi (x) \geq -\varepsilon x, \quad 0 \leq x \leq \delta _{0}. $$
We check that the choice
$$ \delta = - \frac{\varepsilon \delta _{0}^{2}}{2\varphi (r)} $$
works (we omit the trivial case $ \varphi (r) = 0 $). Let $ x \in (0, \delta ] $ be given. We put
$$ a = - \frac{\varepsilon }{2x}, \quad c = - \frac{\varepsilon }{2}x. $$
For $ 0 \leq y \leq \delta _{0} $, using the AG-inequality,
$$ ay^{2} + c = - \frac{\varepsilon }{2} \Big( \frac{y^{2}}{x} + x \Big) \leq -\varepsilon y \leq \varphi (y). $$
For $ \delta _{0} < y \leq r $,
$$ ay^{2} + c \leq ay^{2} = - \frac{\varepsilon }{2x} \cdot y^{2} \leq - \frac{\varepsilon }{2\delta } \cdot \delta _{0}^{2} = \varphi (r) \leq \varphi (y). $$
Hence $ ay^{2} + c \leq \varphi (y) $ for $ 0 \leq y \leq r $, and so
$$ \alpha (x) \geq ax^{2} + c = -\varepsilon x. $$
\end{proof}

\begin{claim} \label{envcl4}
We have $ \alpha (0) = 0 $ and $ \alpha (x) = \varphi (r) $ for $ x \geq r $.
\end{claim}

\begin{proof}
We obtain $ \alpha (0) = 0 $ from Claim \ref{envcl3} and from the fact that $ \alpha $ is non-increasing for $ x \geq 0 $. If we consider $ a = 0 $ and $ c = \varphi (r) $, then $ \varphi (y) \geq \varphi (r) = ay^{2} + c $ for $ 0 \leq y \leq r $. Therefore, $ \alpha (x) \geq ax^{2} + c = \varphi (r) $ for every $ x \in \mathbb{R} $. At the same time, if $ x \geq r $, then $ \alpha (x) \leq \alpha (r) \leq \varphi (r) $.
\end{proof}

\begin{claim} \label{envcl5}
$ \alpha $ is Lipschitz.
\end{claim}

\begin{proof}
By Claims \ref{envcl3} and \ref{envcl4}, there is an $ M > 0 $ such that
$$ \alpha (x) \geq -Mx, \quad x \geq 0. $$
We show that $ \alpha $ is Lipschitz with the constant $ 3M $. It is sufficient to show that
$$ \alpha (y) - \alpha (x) \geq -3M(y - x) \quad \textrm{when } 0 < x \leq y \textrm{ and } y \leq 2x. $$
We have
$$ a_{x}x \geq -M, $$
as $ -Mx \leq \alpha (x) = a_{x}x^{2} + c_{x} \leq a_{x}x^{2} $. Hence,
$$ \alpha (y) - \alpha (x) \geq a_{x}y^{2} + c_{x} - a_{x}x^{2} - c_{x} = a_{x}(y + x)(y - x) \geq 3a_{x}x(y - x) \geq -3M(y - x). $$
\end{proof}

\begin{claim} \label{envcl6}
For every $ 0 \leq p < q \leq r $, we have
$$ \int _{p}^{q} \frac{\varphi (x)}{x^{2}} \, dx \leq \frac{q-p}{q} \cdot \big( -\psi (p) \big) . $$
\end{claim}

\begin{proof}
The formula is valid when $ p = 0 $, as $ \varphi \leq 0 $ and $ \psi (0) = 0 $. When $ p > 0 $,
$$ \int _{p}^{q} \frac{\varphi (x)}{x^{2}} \, dx \leq \int _{p}^{q} \frac{\varphi (p)}{x^{2}} \, dx = \varphi (p) \Big( \frac{1}{p} - \frac{1}{q} \Big) = \frac{q - p}{q} \cdot \frac{\varphi (p)}{p} = \frac{q - p}{q} \cdot \big( -\psi (p) \big) . $$
\end{proof}

\begin{claim} \label{envcl7}
Let numbers $ 0 \leq p < q \leq r $ have the property that $ \alpha (p) = \varphi (p), \alpha (q) = \varphi (q) $ and $ \alpha (x) < \varphi (x) $ for $ p < x < q $. Then
$$ - \int _{p}^{q} \frac{\alpha (x)}{x^{2}} \, dx \leq \frac{q-p}{q} \cdot \big( \psi (q) + \psi (p) \big) . $$
\end{claim}

\begin{proof}
Let $ a $ and $ c $ be the numbers such that
$$ ap^{2} + c = \alpha (p) \quad \textrm{and} \quad aq^{2} + c = \alpha (q), $$
i.e.,
$$ a = \frac{\alpha (q) - \alpha (p)}{q^{2} - p^{2}}, \quad c = \frac{q^{2} \alpha (p) - p^{2} \alpha (q)}{q^{2} - p^{2}}. $$
Let us prove that
$$ \alpha (x) \geq ax^{2} + c, \quad x \in \mathbb{R}. $$
We need to check that $ c = \bar c $ where $ \bar c $ is the greatest number such that
$$ \varphi (y) \geq ay^{2} + \bar c, \quad 0 \leq y \leq r. $$
Suppose the opposite, i.e., $ \bar c < c $. There is a point $ x $ such that $ \varphi (x) = ax^{2} + \bar c $ (due to our assumption that $ \varphi $ is lower semi-continuous). There are four possibilities, and we verify that none of them is possible.

(a) If $ p = 0 $ and $ x = 0 $, then $ 0 = \varphi (0) = a \cdot 0^{2} + \bar c = \bar c < c = 0 $ by Claim~\ref{envcl4}.

(b) If $ p > 0 $ and $ 0 \leq x \leq p $, then we just need to show that $ a_{p} \leq a $, since then we can compute
$$ \varphi (x) \geq a_{p}x^{2} + c_{p} = a_{p}p^{2} + c_{p} - a_{p}(p^{2} - x^{2}) \geq $$
$$ \geq \alpha (p) - a(p^{2} - x^{2}) = ax^{2} + c > ax^{2} + \bar c = \varphi (x). $$
It follows from
$$ \alpha (p) = a_{p}p^{2} + c_{p}, \quad \alpha (q) \geq a_{p}q^{2} + c_{p}, $$
that
$$ a(q^{2} - p^{2}) = \alpha (q) - \alpha (p) \geq a_{p}(q^{2} - p^{2}). $$

(c) If $ q \leq x \leq r $, then we just need to show that $ a \leq a_{q} $, since then we can compute
$$ \varphi (x) \geq a_{q}x^{2} + c_{q} = a_{q}q^{2} + c_{q} + a_{q}(x^{2} - q^{2}) \geq $$
$$ \geq \alpha (q) + a(x^{2} - q^{2}) = ax^{2} + c > ax^{2} + \bar c = \varphi (x). $$
It follows from
$$ \alpha (q) = a_{q}q^{2} + c_{q}, \quad \alpha (p) \geq a_{q}p^{2} + c_{q}, $$
that
$$ a(q^{2} - p^{2}) = \alpha (q) - \alpha (p) \leq a_{q}(q^{2} - p^{2}). $$

(d) If $ p < x < q $, then
$$ \alpha (x) < \varphi (x) = ax^{2} + \bar c \leq \alpha (x). $$

So, $ c = \bar c $ indeed. Now, assuming $ p > 0 $, we arrive at
\begin{eqnarray*}
\int _{p}^{q} \frac{\alpha (x)}{x^{2}} \, dx & \geq & \int _{p}^{q} \Big( a + \frac{c}{x^{2}} \Big) dx = a(q - p) + c\Big( \frac{1}{p} - \frac{1}{q} \Big) \\
 & = & \frac{\alpha (q) - \alpha (p)}{q^{2} - p^{2}} \cdot (q - p) + \frac{q^{2} \alpha (p) - p^{2} \alpha (q)}{q^{2} - p^{2}} \cdot \frac{q - p}{pq} \\
 & = & \frac{q-p}{q+p} \bigg[ \frac{\alpha (q)}{q} + \frac{\alpha (p)}{p} \bigg] = \frac{q-p}{q+p} \bigg[ \frac{\varphi (q)}{q} + \frac{\varphi (p)}{p} \bigg] \\
 & \geq & \frac{q-p}{q} \bigg[ \frac{\varphi (q)}{q} + \frac{\varphi (p)}{p} \bigg] = - \frac{q-p}{q} \big( \psi (q) + \psi (p) \big) .
\end{eqnarray*}
Assuming $ p = 0 $, we obtain $ a = \alpha (q)/q^{2}, \, c = 0 $ from Claim \ref{envcl4} and compute
$$ \int _{p}^{q} \frac{\alpha (x)}{x^{2}} \, dx \geq \int _{p}^{q} a \, dx = aq = \frac{\alpha (q)}{q} = - \psi (q) = - \frac{q-p}{q} \big( \psi (q) + \psi (p) \big) . $$
\end{proof}

\begin{claim} \label{envcl8}
We have
$$ \int _{0}^{\infty } \frac{\alpha (x)}{x^{2}} \, dx \geq 2 \frac{\varphi (r)}{r} + 3 \int _{0}^{r} \frac{\varphi (x)}{x^{2}} \, dx. $$
\end{claim}

\begin{proof}
We may assume that
$$ \int _{0}^{r} \frac{\varphi (x)}{x^{2}} \, dx > - \infty . $$
Let us denote
$$ V_{+} \psi = \sup \sum _{i=1}^{n} \Big( \psi (x_{i}) - \psi (x_{i-1}) \Big) _{+}, $$
$$ V_{-} \psi = \sup \sum _{i=1}^{n} \Big( \psi (x_{i}) - \psi (x_{i-1}) \Big) _{-}, $$
where the supremum is taken over all partitions $ 0 = x_{0} < x_{1} < \dots < x_{n} = r $. For $ 0 < p < q \leq r $, we have
$$ \int _{p}^{q} \frac{\varphi (x)}{x^{2}} \, dx \leq \int _{p}^{q} \frac{\varphi (p)}{x^{2}} \, dx = \varphi (p) \Big( \frac{1}{p} - \frac{1}{q} \Big) \leq \frac{\varphi (p)}{p} - \frac{\varphi (q)}{q} = - \psi (p) + \psi (q). $$
That is,
$$ - \big( \psi (q) - \psi (p) \big) \leq - \int _{p}^{q} \frac{\varphi (x)}{x^{2}} \, dx, \quad 0 \leq p < q \leq r $$
(the formula is valid also when $ p = 0 $ due to the assumption $ \varphi (x)/x \rightarrow 0 $). Given a partition $ 0 = x_{0} < x_{1} < \dots < x_{n} = r $,
$$ \sum _{i=1}^{n} \Big( \psi (x_{i}) - \psi (x_{i-1}) \Big) _{-} \leq - \sum _{i=1}^{n} \int _{x_{i-1}}^{x_{i}} \frac{\varphi (x)}{x^{2}} \, dx = - \int _{0}^{r} \frac{\varphi (x)}{x^{2}} \, dx. $$
Thus,
$$ V_{-} \psi \leq - \int _{0}^{r} \frac{\varphi (x)}{x^{2}} \, dx $$
and
$$ V_{+} \psi = V_{-} \psi + \psi (r) - \psi (0) \leq \psi (r) - \int _{0}^{r} \frac{\varphi (x)}{x^{2}} \, dx. $$

Now, let us consider the set
$$ G = \{ x \in [0, r] : \alpha (x) < \varphi (x) \} . $$
This set is open due to Claims \ref{envcl4}, \ref{envcl5} and our assumption that $ \varphi $ is lower semi-continuous. Thus, we can write
$$ G = \bigcup _{i} (p_{i}, q_{i}) $$
where the intervals $ (p_{i}, q_{i}) $ are pairwise disjoint. Applying Claims \ref{envcl6} and \ref{envcl7}, we can compute
\begin{eqnarray*}
- \int _{0}^{r} \frac{\alpha (x)}{x^{2}} \, dx & = & - \int _{0}^{r} \frac{\varphi (x)}{x^{2}} \, dx + \int _{G} \frac{\varphi (x)}{x^{2}} \, dx - \int _{G} \frac{\alpha (x)}{x^{2}} \, dx \\
 & \leq & - 2\int _{0}^{r} \frac{\varphi (x)}{x^{2}} \, dx + 2\int _{G} \frac{\varphi (x)}{x^{2}} \, dx - \int _{G} \frac{\alpha (x)}{x^{2}} \, dx \\
 & = & - 2\int _{0}^{r} \frac{\varphi (x)}{x^{2}} \, dx + 2\sum _{i} \int _{p_{i}}^{q_{i}} \frac{\varphi (x)}{x^{2}} \, dx - \sum _{i} \int _{p_{i}}^{q_{i}} \frac{\alpha (x)}{x^{2}} \, dx \\
 & \leq & - 2\int _{0}^{r} \frac{\varphi (x)}{x^{2}} \, dx + \sum _{i} \frac{q_{i} - p_{i}}{q_{i}} \big( - 2\psi (p_{i}) + \psi (q_{i}) + \psi (p_{i}) \big) \\
 & \leq & - 2\int _{0}^{r} \frac{\varphi (x)}{x^{2}} \, dx + \sum _{i} \Big( \psi (q_{i}) - \psi (p_{i}) \Big) _{+} \\
 & \leq & - 2\int _{0}^{r} \frac{\varphi (x)}{x^{2}} \, dx + V_{+} \psi \\
 & \leq & - 2\int _{0}^{r} \frac{\varphi (x)}{x^{2}} \, dx + \psi (r) - \int _{0}^{r} \frac{\varphi (x)}{x^{2}} \, dx.
\end{eqnarray*}
It remains just to realize that, due to Claim \ref{envcl4},
$$ \int _{r}^{\infty } \frac{\alpha (x)}{x^{2}} \, dx = \varphi (r) \cdot \int _{r}^{\infty } \frac{1}{x^{2}} \, dx = \varphi (r) \cdot \frac{1}{r}. $$
\end{proof}

\section{Sufficient conditions}\label{sec7}

In this section, we combine Lemma \ref{extension} with Lemma \ref{envelope} and obtain a general extension result. The basic idea of its proof is that we touch a function $ g $ from below by the trace of a separately convex function at every point of a dense set. Taking the supremum of extensions of those functions, we obtain a separately convex function with $ g $ as the trace. The only thing we need to take care of is that the supremum is finite at every point.

\begin{proposition} \label{sufficientprop}
Let $ g : [a, b] \rightarrow \mathbb{R} $ be a Lipschitz function. For every differentiability point $ u $ of $ g $, let us denote
$$ \varphi _{u}(x) = \min \Big\{ g(u + t) - g(u) - g'(u)t : |t| \leq x, u + t \in [a, b] \Big\} , \quad x \geq 0. $$
If there are constants $ K, \varepsilon > 0 $ and a dense subset $ D \subset [a, b] $ consisting of differentiability points of $ g $ such that
$$ \int _{0}^{\varepsilon } \frac{\varphi _{u}(x)}{x^{2}} \, dx \geq -K, \quad u \in D, $$
then there is a separately convex function $ f : \mathbb{R}^{2} \rightarrow \mathbb{R} $ such that $ f(u, u) = g(u) $ for each $ u \in [a, b] $.
\end{proposition}

\begin{proof}
For every $ u \in D $, we define
$$ \alpha _{u}(x) = \sup \Big\{ ax^{2} + c : a, c \leq 0, ay^{2} + c \leq \varphi _{u}(y) \textrm{ for } 0 \leq y \leq b - a \Big\} , \quad x \in \mathbb{R}. $$
We notice that
\begin{equation} \label{comp20} u + t \in [a, b] \quad \Rightarrow \quad \alpha _{u}(t) \leq \varphi _{u}(|t|) \leq g(u + t) - g(u) - g'(u)t. \end{equation}
By Lemma \ref{envelope},
\begin{itemize}
\item $ \alpha _{u} $ is Lipschitz and $ \alpha _{u}(x)/x \rightarrow 0 $ as $ x \rightarrow 0 $,
\item $ \alpha '_{u}(x)/x $ is non-decreasing on its domain for $ x > 0 $,
\item $ \int _{0}^{\infty } \frac{\alpha _{u}(x)}{x^{2}} \, dx \geq 2 \frac{\varphi _{u}(b-a)}{b-a} + 3 \int _{0}^{b-a} \frac{\varphi _{u}(x)}{x^{2}} \, dx $.
\end{itemize}
Let $ L $ be a Lipschitz constant of $ g $. As $ \varphi _{u}(x) \geq -2Lx $ for $ x \geq 0 $, the assumption of the proposition implies that there is a $ K' > 0 $ such that
$$ \int _{0}^{b - a} \frac{\varphi _{u}(x)}{x^{2}} \, dx \geq -K', \quad u \in D. $$
Considering the constant $ C = 4L + 3K' $, we obtain
\begin{equation} \label{comp21} \int _{0}^{\infty } \frac{\alpha _{u}(x)}{x^{2}} \, dx \geq -C, \quad u \in D. \end{equation}

For every $ u \in D $, we further define functions $ \beta _{u} : \mathbb{R} \rightarrow \mathbb{R} $ and $ F_{u} : \mathbb{R}^{2} \rightarrow \mathbb{R} $ by
$$ \beta _{u}(x) = -x \int _{0}^{x} \frac{\alpha _{u}(t)}{t^{2}} \, dt, $$
$$ F_{u}(x,y) = \left\{\begin{array}{ll}
\frac{1}{2x} [(x+y)\alpha _{u}(x) + (x-y)\beta _{u}(x)], & \quad |x| \geq |y|, \, x \neq 0, \\
\frac{1}{2y} [(y+x)\alpha _{u}(y) + (y-x)\beta _{u}(y)], & \quad |x| < |y|, \\
0, & \quad x = y = 0.
\end{array} \right. $$
The function $ F_{u} $ is separately convex due to Lemma \ref{extension}. Using (\ref{comp21}), we can write
$$ \beta _{u}(x) = -x \int _{0}^{x} \frac{\alpha _{u}(t)}{t^{2}} \, dt \leq -x \int _{0}^{\infty } \frac{\alpha _{u}(t)}{t^{2}} \, dt \leq Cx, \quad x \geq 0, \; u \in D. $$
Since $ \alpha _{u}(x) \leq 0 \leq Cx $ for $ x \geq 0 $ at the same time, it follows that
\begin{equation} \label{comp22} F_{u}(x, y) \leq C \max \{ |x|, |y| \} , \quad x, y \in \mathbb{R},\;u\in D. \end{equation}

Finally, we put
$$ f_{u}(x, y) = F_{u}(x - u, y - u) + g(u) + g'(u) \cdot \frac{1}{2} (x + y - 2u), \quad x, y \in \mathbb{R}, \;u \in D,$$
and
\begin{equation} \label{comp23} f(x, y) = \sup _{u \in D} f_{u}(x, y), \quad x, y \in \mathbb{R}. \end{equation}
For $ u \in D $, the function $ f_{u} $ is separately convex, as $ F_{u} $ is separately convex. Due to (\ref{comp22}), we have
$$ f_{u}(x, y) \leq C \max \{ |x - u|, |y - u| \} + g(u) + L \cdot \frac{1}{2} |x + y - 2u|. $$
If $ x $ and $ y $ are fixed, then the values on the right hand side are bounded. Thus, $ f $ is a well-defined separately convex function.

It remains to show that $ f(v, v) = g(v) $ for each $ v \in [a, b] $. Given $ v \in [a, b] $ and $ u \in D $, we denote $ t = v - u $ and use (\ref{comp20}) to compute
$$ f_{u}(v, v) = F_{u}(t, t) + g(u) + g'(u) t = \alpha _{u}(t) + g(u) + g'(u) t \leq g(v). $$
It follows that $ f(v, v) \leq g(v) $ for $ v \in [a, b] $. It is sufficient to check the opposite inequality $ f(v, v) \geq g(v) $ for the elements of a dense subset of $ [a, b] $ only. For $ u \in D $, we have
$$ f(u, u) \geq f_{u}(u, u) = F_{u}(0, 0) + g(u) + g'(u) \cdot 0 = g(u). $$
This completes the proof of the proposition.
\end{proof}

\begin{theorem} \label{sufficientthm}
Let $ g : \mathbb{R} \rightarrow \mathbb{R} $ be a locally Lipschitz function. For every differentiability point $ u $ of $ g $, let us denote
$$ \varphi _{u}(x) = \min \Big\{ g(u + t) - g(u) - g'(u)t : |t| \leq x \Big\} , \quad x \geq 0. $$
If there are an $ \varepsilon > 0 $ and a dense subset $ D \subset \mathbb{R} $ consisting of differentiability points of $ g $ such that the function
$$ u \in D \; \mapsto \; \int _{0}^{\varepsilon } \frac{\varphi _{u}(x)}{x^{2}} \, dx $$
is bounded on every bounded subset of $ D $, then there is a separately convex function $ f : \mathbb{R}^{2} \rightarrow \mathbb{R} $ such that $ f(u, u) = g(u) $ for each $ u \in \mathbb{R} $.
\end{theorem}

\begin{proof}
It is sufficient to apply Propositions \ref{sufficientprop} and \ref{locglob}.
\end{proof}

\begin{corollary} \label{semiconcave}
Let a function $ g : \mathbb{R} \rightarrow \mathbb{R} $ be locally semi-concave with a linear modulus. Then the following assertions are equivalent:

{\rm (i)} There is a separately convex function $ f : \mathbb{R}^{2} \rightarrow \mathbb{R} $ such that $ f(u, u) = g(u) $ for each $ u \in \mathbb{R} $.

{\rm (ii)} The function
$$ x \; \mapsto \; \int _{0}^{1} \frac{\omega _{g}(x, t)}{t^2} \, dt $$
is locally bounded from below.
\end{corollary}

\begin{proof} The implication $ \mathrm{(i)} \Rightarrow \mathrm{(ii)} $ is valid for every locally Lipschitz function (Proposition~\ref{propdiff} and Remark~\ref{remint}). We prove the opposite implication $ \mathrm{(ii)} \Rightarrow \mathrm{(i)} $ for a concave $ g $ only. This is allowed by Lemma~\ref{smoothcase2}, as the functions which are locally semi-concave with a linear modulus are exactly the ones which can be expressed as a sum of a concave function and a $ \C ^{2} $-function.

For a concave function $ g $ which satisfies $ \mathrm{(ii)} $, let us show that the property from Theorem~\ref{sufficientthm} is met. Let $ D $ be the (necessarily dense) set of all differentiability points of $ g $. Due to the concavity, for every $ u \in D $, we have
\begin{eqnarray*}
\varphi _{u}(x) & = & \min \Big\{ g(u + t) - g(u) - g'(u)t : t = \pm x \Big\} \\
 & \geq & \Big( g(u + x) - g(u) - g'(u)x \Big) + \Big( g(u - x) - g(u) - g'(u)(-x) \Big) \\
 & = & \omega _{g}(u, x),
\end{eqnarray*}
and the implication follows.
\end{proof}

\begin{corollary} \label{locsemiconvex}
Let a function $ g : \mathbb{R} \rightarrow \mathbb{R} $ have the property that every bounded interval $ I $ admits a modulus $ \omega = \omega _{I} $ with
$$ \int _{0}^{1} \frac{\omega (t)}{t} \; dt < \infty $$
such that $ g|_{I} $ is semi-convex with the modulus $ \omega $.
Then there is a separately convex function $ f : \mathbb{R}^{2} \rightarrow \mathbb{R} $ such that $ f(u, u) = g(u) $ for each $ u \in \mathbb{R} $.
\end{corollary}

\begin{proof}
Proposition \ref{locglob} allows us to assume that $ g $ is semi-convex on the whole line with a suitable modulus $ \omega = \omega _{\mathbb{R}} $. We may suppose that $ \omega $ is continuous (see e.g. \cite[Corollary 3.6]{DZ1}). Let us define
$$ \Omega (t) = \int _{0}^{t} \omega (s) \, ds, \quad t \geq 0. $$
Notice that
\begin{equation} \label{semi1} \int_0^1\frac{\Omega(t)}{t^2}\;dt=\left[-\frac{\Omega(t)}{t}\right]_0^1+\int_0^1\frac{\Omega'(t)}{t}\;dt=-\Omega(1)+\int_0^1\frac{\omega(t)}{t}\;dt<\infty. \end{equation}
By \cite[Proposition~2.8]{DZ2}, we can write
\begin{equation} \label{semi2} g(z + t) - g(z) - tg'_{+}(z) \geq -2\Omega (|t|), \quad z, t \in \mathbb{R}. \end{equation}
Now, considering the function $ \varphi _{u} $ from Theorem~\ref{sufficientthm} for a differentiability point $ u $ of $ g $, we obtain
$$ \varphi _{u}(x) \geq -2\Omega (x), \quad x \geq 0, $$
and it is sufficient to use (\ref{semi1}).
\end{proof}

\begin{remark}
There is a proof of Theorem~\ref{semiconvex} which is more natural in a manner and does not need the machinery of Section~\ref{sec6}. Let us briefly sketch the construction of $ f $. The modulus can be chosen so that it satisfies, among the continuity, that $ \omega (t)/t $ is non-increasing on $ (0, \infty ) $ and constant on $ [p, \infty ) $ for some $ p > 0 $ (see e.g. \cite[Corollary 3.6]{DZ1}). We pick
$$ \alpha (x) = - 2\Omega (|x|), \quad x \in \mathbb{R}. $$
The assumptions of Lemma~\ref{extension} are fulfilled by (\ref{semi1}) and the additional properties of $ \omega $. Let $ \beta $ and $ F $ be as in Lemma~\ref{extension} and let $ f $ be defined by
$$ f(x, y) = \sup _{u \in \mathbb{R}} \Big( F(x - u, y - u) + g(u) + g'_{+}(u) \cdot \frac{1}{2} (x + y - 2u) \Big) , \quad x, y \in \mathbb{R}, $$
(cf. with (\ref{comp23})). One can show that the function is well-defined using Remark~\ref{extensionbound} and (\ref{semi2}). Moreover, $ f(u, u) = g(u) $ for each $ u \in \mathbb{R} $, which is a consequence of (\ref{semi2}).
\end{remark}

\section{A modification of the extension method}\label{sec007}

Although the extension methods developed in previous sections can be applied to a reasonably general class of functions, there is a natural group of traces that does not seem to be covered. It should be also noted that, in the view of Corollary~\ref{extensmodifcor}, there is no analogue of Theorem~\ref{necess2} which, instead of the symmetric difference, would consider the values of a function only on one side.

\begin{proposition} \label{extensmodif}
Let $ \gamma : (c, \infty ) \to \mathbb{R} $ be a three times differentiable function such that
\begin{equation}\label{derivaceineq}
\gamma \geq 0, \quad \gamma ' \leq 0, \quad \gamma '' \geq 0, \quad \gamma ''' \leq 0 
\end{equation}
on $ (c, \infty ) $. Then the function $ f : (-e^{-c}, e^{-c})^{2} \rightarrow \mathbb{R} $ given by
$$ f(x,y) = \left\{\begin{array}{ll}
- \gamma (\log \frac{1}{x}) \cdot (2y-x) & \\
 \quad \quad - \gamma ' (\log \frac{1}{x}) \cdot \frac{1}{2x}(x-y)(2x-y), & \quad e^{-c} > x \geq |y|, \, x \neq 0, \\
- \gamma (\log (-\frac{1}{x})) \cdot (2x-y) &\\
 \quad \quad - \gamma ' (\log (-\frac{1}{x})) \cdot \frac{3}{2} (y-x), & \quad -e^{-c} < x \leq -|y|, \, x \neq 0, \\
0, & \quad x = y = 0, \\
f(y,x), & \quad \textrm{in the remaining cases,}
\end{array} \right. $$
is separately convex.
\end{proposition}

\begin{proof}
The proof is a straightforward computation provided in two steps.
Firstly we prove the separate convexity only in the quadrants $Q_1:=\{x>|y|\}$, $Q_2:=\{y>|x|\}$, $Q_3:=\{x<-|y|\}$ and $Q_4:=\{y<-|x|\}$, simply by proving the non-negativity of the (unmixed) second partial derivatives.
Secondly we verify that the diagonals $x=y$ and $x=-y$ do not spoil the convexity by observing that $f$ is continuous and then by computing the first derivatives on these diagonals.
Due to the symmetry we need to prove the convexity only on the lines parallel to the $x$-axis.

Define functions 
$$
A_1(x,y):=- \gamma \Big( \log \frac{1}{x} \Big) \cdot (2y-x) - \gamma ' \Big( \log \frac{1}{x} \Big) \cdot \frac{1}{2x}(x-y)(2x-y),
$$
$$
A_3(x,y):=- \gamma \Big( \log \big( -\frac{1}{x} \big) \Big) \cdot (2x-y) - \gamma ' \Big( \log \big( -\frac{1}{x} \big) \Big) \cdot \frac{3}{2} (y-x),
$$
$$
A_2(x,y):=A_1(y,x) \quad\text{and}\quad A_4(x,y):=A_3(y,x). \phantom{\Big( \Big)}
$$
Then $f(x,y)=A_i(x,y)$ for $(x,y)\in Q_i$, $i=1,\dots, 4$, whenever one of the functions is defined.

First we compute the partial derivatives.
A simple computation shows that
\begin{equation}\label{1dA1}
\begin{aligned}
\frac{\partial A_1}{\partial x}(x,y)=\frac{1}{2x^2} \bigg[ 2x^2\gamma \Big( & \log \frac{1}{x} \Big) + (y^2-4x^2+4xy)\gamma ' \Big( \log \frac{1}{x} \Big) \\
 & +(2x^2-3xy+y^2)\gamma '' \Big( \log \frac{1}{x} \Big) \bigg] ,
\end{aligned}
\end{equation}
\begin{equation}\label{1dA2}
\frac{\partial A_2}{\partial x}(x,y)=\frac{1}{2y} \bigg[ -4y\gamma \Big( \log\frac{1}{y} \Big) +(3y-2x)\gamma' \Big( \log\frac{1}{y} \Big) \bigg] ,
\end{equation}
\begin{equation}\label{1dA3}
\begin{aligned}
\frac{\partial A_3}{\partial x}(x,y)=\frac{1}{2x} \bigg[ -4x\gamma \Big( \log & \big( -\frac{1}{x} \big) \Big) +(7x-2y)\gamma' \Big( \log \big( -\frac{1}{x} \big) \Big) \\
 & +3(y-x)\gamma'' \Big( \log \big( -\frac{1}{x} \big) \Big) \bigg] ,
\end{aligned}
\end{equation}
\begin{equation}\label{1dA4}
\frac{\partial A_4}{\partial x}(x,y)=\gamma \Big( \log \big( -\frac{1}{y} \big) \Big) -\frac{3}{2}\gamma' \Big( \log \big( -\frac{1}{y} \big) \Big) ,
\end{equation}
\begin{equation}\label{2dA1}
\begin{aligned}
\frac{\partial ^2 A_1}{\partial x^2}(x,y)=\frac{1}{2x^3} \bigg[ -2(x+y)^2\gamma ' \Big( & \log \frac{1}{x} \Big) + (x-y)(4x+3y)\gamma '' \Big( \log \frac{1}{x} \Big) \\
 & +(x-y)(y-2x)\gamma ''' \Big( \log \frac{1}{x} \Big) \bigg] ,
\end{aligned}
\end{equation}
\begin{equation}\label{2dA2}
\frac{\partial ^2 A_2}{\partial x^2}(x,y)=-\frac{1}{y} \gamma' \Big( \log \frac{1}{y} \Big) ,
\end{equation}
\begin{equation}\label{2dA3}
\begin{aligned}
\frac{\partial ^2 A_3}{\partial x^2}(x,y)=\frac{1}{2x^2} \bigg[ 2(2x+y) \gamma ' \Big( & \log \big( -\frac{1}{x} \big) \Big) -(7x+y) \gamma '' \Big( \log \big( -\frac{1}{x} \big) \Big) \\
 & +3(x-y) \gamma ''' \Big( \log \big( -\frac{1}{x} \big) \Big) \bigg] ,
\end{aligned}
\end{equation}
\begin{equation}\label{2dA4}
\frac{\partial ^2 A_4}{\partial x^2}(x,y)=0.
\end{equation}
Now, $(x+y)^2\geq 0$ for any $x,y$ and $x>|y|$ implies $(x-y)(4x+3y)\geq 0$ and $(x-y)(y-2x)\leq 0$.
Hence, using \eqref{2dA1}, we obtain that $\frac{\partial ^2 f}{\partial x^2}(x,y)\geq 0$ whenever $x>|y|$.
Similarly, $\frac{\partial ^2 f}{\partial x^2}(x,y)\geq 0$ whenever $y>|x|$ using \eqref{2dA2}.
Using the symmetry of $f$ we obtained that $f$ is separately convex in both $Q_1$ and $Q_2$.

To verify separate convexity in $Q_3$ and $Q_4$ we observe that $x<-|y|$ implies $2x+y\leq 0$, $7x+y\leq 0$ and $x-y\leq 0$, 
which, using \eqref{2dA3}, implies $\frac{\partial ^2 f}{\partial x^2}(x,y)\geq 0$ in $Q_3$ and similarly we get the same result for $Q_4$ using \eqref{2dA4}.

To prove the second part, we note that it is easy to check that $f$ is continuous in the direction of the $x$-axis in the points $(a,\pm a)$, including $(0,0)$. So we just need to observe the following four inequalities.

If $0<x=y$ then (using \eqref{1dA1} and \eqref{1dA2})
$$
\gamma \Big( \log \frac{1}{x} \Big) +\frac{1}{2}\gamma ' \Big( \log \frac{1}{x} \Big) =\frac{\partial A_1}{\partial x}(x,x)\geq \frac{\partial A_2}{\partial x}(x,x)=-2\gamma \Big( \log \frac{1}{x} \Big) +\frac{1}{2}\gamma ' \Big( \log \frac{1}{x} \Big) .
$$
If $0>x=y$ then (using \eqref{1dA3} and \eqref{1dA4})
$$
\begin{aligned}
\gamma \Big( \log \big( -\frac{1}{x} \big) \Big) -\frac{3}{2}\gamma ' \Big( \log \big( -\frac{1}{x} \big) \Big) =\frac{\partial A_4}{\partial x}(x,x)\geq \frac{\partial A_3}{\partial x}(x,x)\\
=-2\gamma \Big( \log \big( -\frac{1}{x} \big) \Big) +\frac{5}{2}\gamma ' \Big( \log \big( -\frac{1}{x} \big) \Big) .
\end{aligned}
$$
If $0<x=-y$ then (using \eqref{1dA1} and \eqref{1dA4})
$$
\begin{aligned}
\gamma \Big( \log \frac{1}{x} \Big) -\frac{7}{2}\gamma ' \Big( \log \frac{1}{x} \Big) & +3\gamma '' \Big( \log \frac{1}{x} \Big) =\frac{\partial A_1}{\partial x}(x,-x)
\geq \frac{\partial A_4}{\partial x}(x,-x)\\
 & =\gamma \Big( \log \frac{1}{x} \Big) -\frac{3}{2}\gamma ' \Big( \log \frac{1}{x} \Big) .
\end{aligned}
$$
Finally, if $0>x=-y$ then (using \eqref{1dA2} and \eqref{1dA3})
$$
\begin{aligned}
-2\gamma \Big( \log \big( -\frac{1}{x} \big) \Big) +\frac{5}{2}\gamma ' \Big( \log \big( -\frac{1}{x} \big) \Big) =\frac{\partial A_2}{\partial x}(x,-x)\geq \frac{\partial A_3}{\partial x}(x,-x)\\
=-2\gamma \Big( \log \big( -\frac{1}{x} \big) \Big) +\frac{9}{2}\gamma ' \Big( \log \big( -\frac{1}{x} \big) \Big) -3\gamma '' \Big( \log \big( -\frac{1}{x} \big) \Big) .
\end{aligned}
$$

The above inequalities reduce to inequalities
$$
3\gamma \Big( \log \frac{1}{x} \Big) \geq 0\quad\text{and}\quad -2\gamma ' \Big( \log \frac{1}{x} \Big) +3\gamma '' \Big( \log \frac{1}{x} \Big) \geq 0
$$
whenever $x>0$ and 
$$
3\gamma \Big( \log \big( -\frac{1}{x} \big) \Big) -4\gamma ' \Big( \log \big( -\frac{1}{x} \big) \Big) \geq 0
$$
and
$$
-2\gamma ' \Big( \log \big( -\frac{1}{x} \big) \Big) +3\gamma '' \Big( \log \big( -\frac{1}{x} \big) \Big) \geq 0
$$
whenever $x<0$, which all hold due to \eqref{derivaceineq}.
\end{proof}

\begin{corollary} \label{extensmodifcor}
Let $ h : [0, a) \to \mathbb{R} $ be a function which is bounded from below such that $ h(0) = 0 $ and $ h'_{+}(0) = 0 $. Then there exists a separately convex function $ f : (-a, a)^{2} \to \mathbb{R} $ such that its trace $ g(t) = f(t, t) $ is an odd function with $ g'(0) = 0 $ and $ g(t) \leq h(t) $ for $ 0 \leq t < a $.
\end{corollary}

To prove the corollary, we need the following observation.

\begin{claim} \label{smoothmajorant}
Let $ \kappa : (c, \infty ) \to \mathbb{R} $ be a function which is bounded from above with $ \kappa (x) \to 0 $ as $ x \to \infty $. Then $ \kappa $ has a majorant $ \gamma : (c, \infty ) \to \mathbb{R} $ of the class $ \mathcal{C}^{\infty } $ so that $ \gamma (x) \to 0 $ as $ x \to \infty $ and
$$ (-1)^{k} \gamma ^{(k)} (x) \geq 0, \quad x > c, \; k = 0, 1, 2, \dots \; . $$
\end{claim}

\begin{proof}
Without loss of generality, we consider the following special case. Let $ c = 0 $ and $ \kappa $ be of the form
$$ \kappa = \sum _{i=1}^{\infty } 2^{-i} \cdot \mathbf{1}_{(0, a_{i}]} $$
where $ a_{i} \geq 1 $ for $ i \in \mathbb{N} $. Let us show that
\begin{equation} \label{omega} \gamma = \sum _{i=1}^{\infty } 2^{-i} \cdot \gamma _{i} \end{equation}
where
$$ \gamma _{i}(x) = \frac{2a_{i}}{x + a_{i}}, \quad x > 0, $$
works. We have
$$ \gamma _{i}^{(k)} (x) = (-1)^{k} k! \cdot \frac{2a_{i}}{(x + a_{i})^{k+1}}. $$
Hence $ | \gamma _{i}^{(k)} (x) | \leq k! \cdot \frac{2a_{i}}{(0 + a_{i})^{k+1}} \leq 2 \cdot k! $, so the derivatives of the partial sums in (\ref{omega}) converge uniformly on $ (0, \infty ) $. We can write
$$ (-1)^{k} \gamma ^{(k)} (x) = (-1)^{k} \sum _{i=1}^{\infty } 2^{-i} \cdot \gamma _{i}^{(k)} (x) = k! \sum _{i=1}^{\infty } 2^{-i} \cdot \frac{2a_{i}}{(x + a_{i})^{k+1}} \geq 0. $$
\end{proof}

\begin{proof}[Proof of Corollary \ref{extensmodifcor}]
Let
$$ \kappa (x) = - \frac{h(e^{-x})}{e^{-x}}, \quad x > c := \log \frac{1}{a}, $$
and let $ \gamma : (c, \infty ) \to \mathbb{R} $ be a majorant given by Claim \ref{smoothmajorant}. Then Proposition \ref{extensmodif} provides us with a separately convex function $ f : (-a, a)^{2} \to \mathbb{R} $ such that
$$ f(t,t) = -f(-t,-t) = - t \cdot \gamma \Big( \log \frac{1}{t} \Big) \quad \textrm{for } 0 < t < a. $$
Its trace $ g(t) = f(t, t) $ is an odd function. Since $ \gamma (x) \to 0 $ as $ x \to \infty $, we obtain $ g'(0) = 0 $. For $ 0 < t < a $, we can write
$$ g(t) = - t \cdot \gamma \Big( \log \frac{1}{t} \Big) \leq - t \cdot \kappa \Big( \log \frac{1}{t} \Big) = h(t). $$
\end{proof}

\section{Examples}\label{sec8}
The first example illustrates that the traces of separately convex functions may not have a property typical for the semi-convex functions, 
the existence of one-sided derivatives at every point (see Figure~\ref{nondiff} in the introduction).

\begin{example}\label{example3}

Define a real function $h$ on $\er$ by $h(x):=\max(1-x^2,0)$.
Then $h$ is a Lipschitz function whose derivative exists everywhere except $-1$ and $1$ and  
$$
\frac{h'(x)}{x}= 
\begin{cases} 
-2& \mbox{if } 0<|x|<1, \\
0 & \mbox{if } |x|>1. 
\end{cases} 
$$
Also,
$$
\int_0^x\frac{h(t)-h(0)}{t^2}\;dt=
\begin{cases} 
-x & \mbox{if } 0<x\leq 1, \\
-2+\frac{1}{x} & \mbox{if } x>1. 
\end{cases}
$$
So using Lemma~\ref{extension} we can obtain a separately convex function $\tilde f$ such that $\tilde f(t,t)=h(t)$ for $t\in\er$ and such that
\begin{equation}\label{odhadex2}
\tilde f(x,y)\leq 1+2\max (|x|,|y|)
\end{equation}
whenever $x,y\in\er$.

For $n\in\en$ define functions $h_n:\er\to\er$ and $f_n:\er^2\to\er$ by $h_n(x):=3^{-n}h(3^nx-2)$ and $f_n(x,y):=3^{-n}\tilde f(3^nx-2,3^ny-2)$.
Consider functions $g:\er\to\er$ and $f:\er^2\to\er$ defined by 
$$
g(x):=\sup_{n}h_n(x)\quad\text{and}\quad f(x,y):=\sup_n f_n(x,y).
$$
First observe that $f$ is finite by \eqref{odhadex2} and therefore  $f$ is a separately convex function on $\er^2$ such that $f(t,t)=g(t)$ for $t\in\er$.
Note that $g=h_n$ on $[3^{-n},3^{-n+1}]$ which in particular means that $g\left(2\cdot 3^{-n}\right)=3^{-n}$ and $g\left(3^{-n}\right)=0$.
Therefore, $g'_+(0)$ does not exist and, in particular, $g$ is not semi-convex with any modulus.

\end{example}

\begin{example}\label{example1}

Define a function $g:(-1,1)\to\er$ by $g(x)=\frac{|x|}{\log|x|}$ for $x\not=0$ and $g(0)=0$. Then $g$ is $\C^1$ on $(-1,1)$ but 
$$
-\int\limits_{0}^{\frac{1}{2}}\frac{\omega_g(0,t)}{t^2}\;dt=-\int\limits_{0}^{\frac{1}{2}}\frac{2}{t\log t}\;dt=\infty.
$$
Note that $g$ is also concave on $(-1,1)$ and so by modifying it outside a neighbourhood of $0$ we can obtain a concave $\C^1$ function on $\er$
which (using e.g. Proposition~\ref{propdiff}) cannot be the trace of a separately convex function.
On the other hand, it might be worth noting that the function $\tilde g$ defined by $\tilde g(x)=\frac{x}{\log|x|}$ (and $\tilde g(0)=0$)
actually can be extended to a separately convex function $\tilde f$ on $(-1,1)^2$ using Proposition~\ref{extensmodif} (applied on $ \gamma(t) = \frac{1}{t} $).
\end{example}

The next example are actually two examples, in the first one we show that there is a function that satisfies the necessary condition from Proposition~\ref{propdiff}, but it is not a trace by Theorem~\ref{necess2}. The second one illustrates that even the condition from Theorem~\ref{necess2} is not sufficient.

\begin{example}\label{example2}

Define functions $\varphi:\er\to\er$ and $\varphi_u:\er\to\er,u>0,$ by
$$
\varphi(x):=
\begin{cases}
-\frac{\cos\left(\pi x\right)+1}{2} & \text{if}\quad |x|\leq 1,\\
0 & \text{if}\quad |x|>1,
\end{cases}
$$
and
$$
\varphi_u(x):=u\varphi\left(\frac{x}{u}\right).
$$
Put $t_i=\frac{1}{2^i}$, $u_i=\frac{1}{i2^i}$ and define functions $\phi_i:\er\to\er$ and $\psi:\er\to\er$  by 
$$
\phi_i(x):=\varphi_{u_i}(x-t_i)\quad\text{and}\quad \psi:=\sum\limits_{i=10}^{\infty}\phi_i.
$$ 

\begin{figure}[h]
\psset{unit=0.45mm}
\begin{pspicture}(-12,-52)(242,42)
\psset{linecolor=gray}
\psset{linewidth=0.6pt}
\psline(-10,0)(240,0)
\psline(0,-50)(0,40)
\psset{linecolor=black}
\psset{linewidth=1pt}
\psline(220,0)(240,0)
\psecurve(160,-20)(180,0)(200,-20)(220,0)(240,-20)
\psline(109.09,0)(180,0)
\psecurve(81.82,-9.09)(90.91,0)(100,-9.09)(109.09,0)(118.18,-9.09)
\psline(54.17,0)(90.91,0)
\psecurve(41.67,-4.17)(45.83,0)(50,-4.17)(54.17,0)(58.33,-4.17)
\psline(26.92,0)(45.83,0)
\psecurve(21.15,-1.92)(23.08,0)(25,-1.92)(26.92,0)(28.85,-1.92)
\psline(13.39,0)(23.08,0)
\psecurve(10.71,-0.89)(11.61,0)(12.5,-0.89)(13.39,0)(14.29,-0.89)
\psline(6.67,0)(11.61,0)
\psecurve(5.32,-0.42)(5.83,0)(6.25,-0.42)(6.67,0)(7.08,-0.42)
\psline(3.325,0)(5.83,0)
\psecurve(2.725,-0.2)(2.925,0)(3.125,-0.2)(3.325,0)(3.525,-0.2)
\psline(-10,0)(2.925,0)
\psset{linewidth=0.4pt}
\psline[linestyle=dashed](200,-35)(200,3)
\psline[linestyle=dashed](180,-35)(180,3)
\psline[linestyle=dashed](220,-35)(220,3)
\psline[linestyle=dashed](100,-17.5)(100,3)
\psline[linestyle=dashed](90.91,-17.5)(90.91,3)
\psline[linestyle=dashed](109.09,-17.5)(109.09,3)
\psline(100,-2)(100,2)
\psline(50,-1.5)(50,2)
\psline(25,-0.5)(25,2)
\psline[linestyle=dashed](165,-20)(225,-20)
\psline[linestyle=dashed](82,-9.09)(112.5,-9.09)
\psline{<->}(170,-20)(170,0)
\psline{<->}(180,-30)(200,-30)
\psline{<->}(200,-30)(220,-30)
\psline{<->}(85,-9.09)(85,0)
\psline{<->}(90.91,-15)(100,-15)
\psline{<->}(100,-15)(109.09,-15)
\put(200,10){\makebox(0,0){$ t_{10} $}}
\put(100,10){\makebox(0,0){$ t_{11} $}}
\put(50,10){\makebox(0,0){$ t_{12} $}}
\put(25,10){\makebox(0,0){$ t_{13} $}}
\put(11,10){\makebox(0,0){$ \dots $}}
\put(160,-10){\makebox(0,0){$ u_{10} $}}
\put(190,-38){\makebox(0,0){$ u_{10} $}}
\put(210,-38){\makebox(0,0){$ u_{10} $}}
\put(78,-5){\makebox(0,0){$ u_{11} $}}
\put(95,-22){\makebox(0,0){$ u_{11} $}}
\put(107.5,-22){\makebox(0,0){$ u_{11} $}}
\end{pspicture}
\caption{The graph of $ \psi $.}
\end{figure}
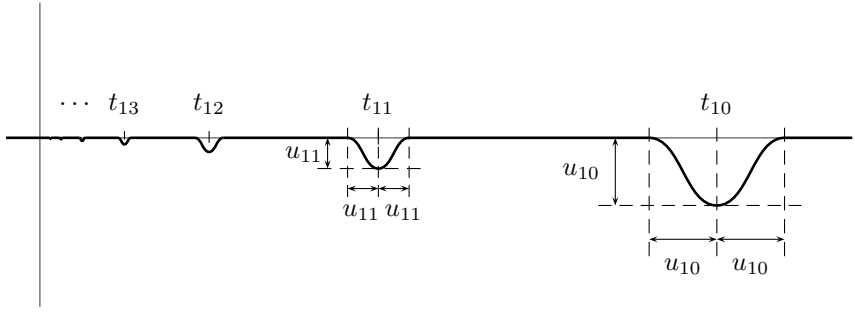

We will prove that there is a constant $C$ such that 
\begin{equation}\label{psi1}
-\int\limits_{0}^{\infty}\frac{\omega_{\psi}(x,s)}{s^2}\;ds\leq C
\end{equation}
for every $x\in\er$.
On the other hand, we will prove that
\begin{equation}\label{psi2}
-\int\limits_{0}^{1}\frac{\omega_{\psi}^*(0,s)}{s^2}\;ds=\infty.
\end{equation}

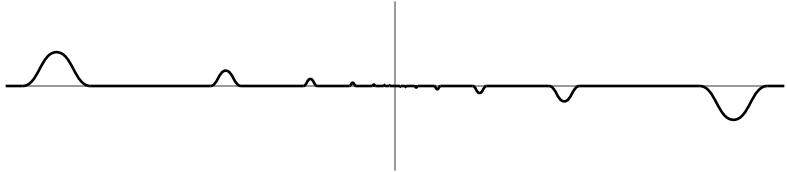
\begin{figure}[h]
\psset{unit=0.225mm}
\begin{pspicture}(-232,-52)(232,52)
\psset{linecolor=gray}
\psset{linewidth=0.6pt}
\psline(-230,0)(230,0)
\psline(0,-50)(0,50)
\psset{linecolor=black}
\psset{linewidth=1pt}
\psline(220,0)(230,0)
\psecurve(160,-20)(180,0)(200,-20)(220,0)(240,-20)
\psline(109.09,0)(180,0)
\psecurve(81.82,-9.09)(90.91,0)(100,-9.09)(109.09,0)(118.18,-9.09)
\psline(54.17,0)(90.91,0)
\psecurve(41.67,-4.17)(45.83,0)(50,-4.17)(54.17,0)(58.33,-4.17)
\psline(26.92,0)(45.83,0)
\psecurve(21.15,-1.92)(23.08,0)(25,-1.92)(26.92,0)(28.85,-1.92)
\psline(13.39,0)(23.08,0)
\psecurve(10.71,-0.89)(11.61,0)(12.5,-0.89)(13.39,0)(14.29,-0.89)
\psline(6.67,0)(11.61,0)
\psecurve(5.32,-0.42)(5.83,0)(6.25,-0.42)(6.67,0)(7.08,-0.42)
\psline(3.325,0)(5.83,0)
\psecurve(2.725,-0.2)(2.925,0)(3.125,-0.2)(3.325,0)(3.525,-0.2)
\psline(-2.925,0)(2.925,0)
\psline(-220,0)(-230,0)
\psecurve(-160,20)(-180,0)(-200,20)(-220,0)(-240,20)
\psline(-109.09,0)(-180,0)
\psecurve(-81.82,9.09)(-90.91,0)(-100,9.09)(-109.09,0)(-118.18,9.09)
\psline(-54.17,0)(-90.91,0)
\psecurve(-41.67,4.17)(-45.83,0)(-50,4.17)(-54.17,0)(-58.33,4.17)
\psline(-26.92,0)(-45.83,0)
\psecurve(-21.15,1.92)(-23.08,0)(-25,1.92)(-26.92,0)(-28.85,1.92)
\psline(-13.39,0)(-23.08,0)
\psecurve(-10.71,0.89)(-11.61,0)(-12.5,0.89)(-13.39,0)(-14.29,0.89)
\psline(-6.67,0)(-11.61,0)
\psecurve(-5.32,0.42)(-5.83,0)(-6.25,0.42)(-6.67,0)(-7.08,0.42)
\psline(-3.325,0)(-5.83,0)
\psecurve(-2.725,0.2)(-2.925,0)(-3.125,0.2)(-3.325,0)(-3.525,0.2)
\end{pspicture}
\caption{The graph of $ \xi $.}
\end{figure}

Moreover, in the view of Example~\ref{example1}, one might also wonder whether a function $\xi:\er\to\er$ defined by
$$
\xi(x):=
\begin{cases}
\psi(x) & \text{if}\quad x\geq 0,\\
-\psi(-x) & \text{if}\quad x<0,
\end{cases}
$$
can be the trace of a separately convex function $f$ on $\er^2$. In this case, we will argue that no such function $f$ exists using Remark~\ref{remmostgeneral}.
The function $\xi$, however, satisfies the necessary condition from Theorem~\ref{necess2}.
Indeed, we will prove that there is a constant $C^*$ such that 
\begin{equation}\label{odhadmaxpriklad}
-\int\limits_{0}^{\infty}\frac{\omega^{*}_{\xi}(x,s)}{s^2}\;ds\leq C^*
\end{equation}
for every $x\in\er$.

We start the proof with the following formulae. 
One can verify that there are constants $ C_{1}, C_{2} $ and $ C_{3} $ such that the following three inequalities hold (independently of $ u $ and $ x $):
\begin{equation}\label{jasnyodhad1}
-\int\limits_{0}^{\infty}\frac{\omega_{\varphi_u}(x,s)}{s^2}\;ds\leq C_{1} \min \left\{ 1, \frac{u^2}{x^2} \right\},
\end{equation}
\begin{equation}\label{jasnyodhad2}
-\int\limits_{0}^{\infty}\frac{\omega^{*}_{\varphi_u}(x,s)}{s^2}\;ds\leq C_{2},
\end{equation}
\begin{equation}\label{jasnyodhad3}
-\int\limits_{0}^{\infty}\frac{\omega^{*}_{-\varphi_u}(x,s)}{s^2}\;ds\leq C_{3}.
\end{equation}

To prove \eqref{psi2}, we can compute
$$
-\int\limits_{0}^{1}\frac{\omega_{\psi}^*(0,s)}{s^2}\;ds\geq \sum\limits_{i=10}^{\infty}\; \int\limits_{t_{i+1}}^{t_i}\frac{u_{i+1}}{s^2}\;ds
=\sum\limits_{i=10}^{\infty} \left(\frac{u_{i+1}}{t_{i+1}}-\frac{u_{i+1}}{t_i}\right)=\sum\limits_{i=10}^{\infty} \frac{1}{2(i+1)}=\infty.
$$
We prove \eqref{psi1} only for the most difficult case $x \in (0, t_{10}) $. We can write
$$ \omega_{\phi_i}(x,s) = \omega_{\varphi_{u_i}}(x-t_{i},s) $$
and
$$ -\int\limits_{0}^{\infty}\frac{\omega_{\phi_i}(x,s)}{s^2}\;ds\leq C_{1} \min \left\{ 1, \frac{u_{i}^2}{(x-t_{i})^2} \right\} . $$
Pick such a $ j \geq 10 $ that $ x \in [t_{j+1}, t_{j}) $. For $ 10 \leq i \leq j-1 $, we can compute $ t_i - x \geq t_{i} - t_{j} \geq t_{i} - t_{i+1} = 2^{-(i+1)} $, and thus
$$ \frac{u_{i}^2}{(x-t_{i})^2} \leq \frac{1}{i^2 2^{2i}} \cdot 2^{2(i+1)} = \frac{4}{i^2}. $$
The integral in \eqref{psi1} is then equal to
$$
\begin{aligned}
- \int\limits_{0}^{\infty}\sum\limits_{i=10}^{j-1} & \frac{\omega_{\phi_i}(x,s)}{s^2}\;ds
- \int\limits_{0}^{\infty}\frac{\omega_{\phi_j}(x,s)}{s^2}\;ds
- \int\limits_{0}^{\infty}\sum\limits_{i=j+1}^{\infty}\frac{\omega_{\phi_i}(x,s)}{s^2}\;ds \\
 & \leq C_1 \sum_{i=10}^{j-1} \frac{u_{i}^2}{(x-t_{i})^2} + C_1 + C_2 \leq C_1 \sum_{i=10}^{\infty} \frac{4}{i^2} + C_1 + C_2 =: C < \infty .
\end{aligned}
$$

Using Remark~\ref{remmostgeneral}, we show that the function $ \xi $ can not be the trace of a separately convex function.
Consider sequences $r_i=t_i=\frac{1}{2^i}$ and $p_i=\frac{3}{2}\cdot\frac{1}{2^i}$, $i\geq 10$.
Note that $\xi(r_i)=-u_i=-\frac{1}{i2^i}$ and since (using $i\geq 10$)
$$
t_i+u_i=\frac{1}{2^i}\left(1+\frac{1}{i}\right)\leq \frac{3}{2}\cdot\frac{1}{2^i}\leq \frac{1}{2^i}\left(2-\frac{2}{i-1}\right)=t_{i-1}-u_{i-1},
$$
we obtain that $\xi(-p_i)=0$.
Formula \eqref{mostgenfor} gives us
$$
\begin{aligned}
\infty>& -\sum_{i=10}^{\infty}\biggl[\xi(r_{i+1})\cdot\frac{p_i-p_{i+1}}{(p_i+r_{i+1})(p_{i+1}+r_{i+1})}+\xi(-p_{i})\cdot\frac{r_i-r_{i+1}}{(p_i+r_{i+1})(p_i+r_i)}\biggr]\\
=&\sum_{i=10}^{\infty}\frac{1}{(i+1)2^{i+1}}\cdot\frac{3}{2}\cdot\frac{\frac{1}{2^i}-\frac{1}{2^{i+1}}}{(\frac{3}{2}\cdot\frac{1}{2^i}+\frac{1}{2^{i+1}})\cdot(\frac{3}{2}\cdot\frac{1}{2^{i+1}}+\frac{1}{2^{i+1}})}\\
=& \sum_{i=10}^{\infty}\frac{1}{(i+1)}\cdot\frac{3}{2}\cdot\frac{1}{(3+1)\cdot(\frac{3}{2}+1)}=\sum_{i=10}^{\infty}\frac{3}{20}\cdot\frac{1}{(i+1)}\;,
\end{aligned}
$$
which is not possible.

It remains to show \eqref{odhadmaxpriklad}. First note that the integral in \eqref{odhadmaxpriklad} is trivially equal to $0$ for $x=0$.
To finish the proof, fix $x\not=0$.
The integral can be divided into two parts
$$
I_1:=-\int\limits_{0}^{|x|}\frac{\omega^{*}_{\xi}(x,t)}{t^2}\;dt\quad\text{and}\quad I_2:=-\int\limits_{|x|}^{\infty}\frac{\omega^{*}_{\xi}(x,t)}{t^2}\;dt.
$$
It remains to find constants $A$ and $B$ independent of $x$ such that $I_1\leq A$ and $I_2\leq B$.
One can show that $ I_1 \leq C_{3} $ for $ x < 0 $ and $ I_1 \leq 3C_{2} $ for $ x > 0 $.
Let $L$ be a Lipschitz constant of $\xi$ (we can take a Lipschitz constant of $ \varphi $).
Then we can estimate
$$
\begin{aligned}
\omega_{\xi}(x,t)=&\;\xi(x+t)+\xi(x-t)-2\xi(x)=\xi(t+x)-\xi(t-x)-2\xi(x)\\
\geq& -2L|x|-2\xi(x)\geq -4L|x|.
\end{aligned}
$$
Moreover, since the function $t\mapsto -4L|x|$ is non-increasing, we also obtain that $\omega^{*}_{\xi}(x,t)\geq-4L|x|$.
Now, we can write
$$
I_2\leq 4L|x|\int\limits_{|x|}^{\infty}\frac{1}{t^2}\;dt=\frac{4L|x|}{|x|}=4L,
$$
and we are done.

\end{example}

The last example illustrates that there is a concave function which is a trace, but does not satisfy the sufficient condition from Theorem~\ref{semiconvex}.
The idea of the construction is based on the observation that the integrals in condition (ii) from Theorem~\ref{concave} do not need to converge (locally) uniformly.

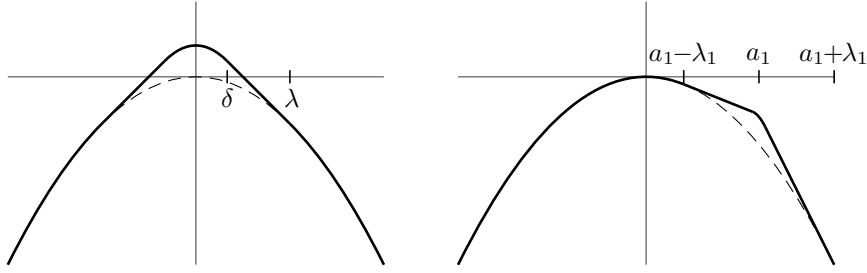
\begin{figure}[t]
\psset{unit=0.25mm}
\begin{pspicture}(-100,-100)(100,40)
\psset{linecolor=gray}
\psset{linewidth=0.6pt}
\psline(-100,0)(100,0)
\psline(0,-100)(0,40)
\psset{linecolor=black}
\psline(50,-3.5)(50,3.5)
\psline(16.66,-3.5)(16.66,3.5)
\psset{linewidth=0.4pt}
\parabola[linestyle=dashed](100,-100)(0,0)
\psset{linewidth=1pt}
\psbezier(50,-25)(66.66,-41.66)(83.33,-66.66)(100,-100)
\psbezier(-50,-25)(-66.66,-41.66)(-83.33,-66.66)(-100,-100)
\psline(50,-25)(16.66,8.33)
\psline(-50,-25)(-16.66,8.33)
\parabola(16.66,8.33)(0,16.66)
\put(16.66,-11){\makebox(0,0){$ \delta $}}
\put(51,-11){\makebox(0,0){$ \lambda $}}
\end{pspicture}
\hspace{7.5mm}
\begin{pspicture}(-100,-100)(100,40)
\psset{linecolor=gray}
\psset{linewidth=0.6pt}
\psline(-100,0)(100,0)
\psline(0,-100)(0,40)
\psset{linecolor=black}
\psline(20,-3.5)(20,3.5)
\psline(60,-3.5)(60,3.5)
\psline(100,-3.5)(100,3.5)
\psset{linewidth=0.4pt}
\parabola[linestyle=dashed](100,-100)(0,0)
\psset{linewidth=1pt}
\psbezier(-100,-100)(-60,-20)(-20,12)(20,-4)
\psline(20,-4)(56.72,-18.69)
\psline(100,-100)(63.28,-26.57)
\psbezier(56.72,-18.69)(58.91,-19.56)(61.09,-22.19)(63.28,-26.57)
\put(60,11){\makebox(0,0){$ a_{1} $}}
\put(20,12.5){\makebox(0,0){$ a_{1} \hspace{-3pt} - \hspace{-3pt} \lambda _{1} $}}
\put(100,12.5){\makebox(0,0){$ a_{1} \hspace{-3pt} + \hspace{-3pt} \lambda _{1} $}}
\end{pspicture}
\caption{(a) The function $ g_{\delta,\lambda} $ for the choice $ \delta = \frac{1}{6} $ and $ \lambda = \frac{1}{2} $.
(b) The function $ g_{1} $ for the choice $ a_{1} = \frac{3}{5} $ and $ \lambda_{1} = \frac{2}{5} $.}
\label{gecka}
\end{figure}

\begin{example}\label{example4}

For $1>\lambda\geq\delta>0$ define a function $g_{\delta,\lambda}:\er\to\er$ by (cf. Figure~\ref{gecka}(a))
$$
g_{\delta,\lambda}(x):= 
\begin{cases} 
-\frac{\lambda}{\delta}x^2+\lambda^2-\lambda\delta &\mbox{if } |x| \leq \delta, \\
-2\lambda|x|+\lambda^2 &\mbox{if } \delta<|x|\leq\lambda, \\
-x^2 & \mbox{if } |x|>\lambda. 
\end{cases} 
$$

The function $g_{\delta,\lambda}$ is even and an easy computation gives that $g'_{\delta,\lambda}(x)$ exists for $x>0$ and moreover
$$
\frac{g'_{\delta,\lambda}(x)}{x}= 
\begin{cases} 
-\frac{2\lambda}{\delta} &\mbox{if } 0<x \leq \delta, \\
-\frac{2\lambda}{x} &\mbox{if } \delta<x\leq\lambda, \\
-2 & \mbox{if } x>\lambda. 
\end{cases} 
$$
Therefore the function $x\mapsto\frac{g'_{\delta,\lambda}(x)}{x}$ is clearly non-decreasing.
It is easy to see that $g_{\delta,\lambda}$ is also always concave.
Moreover, for $x\geq \lambda,$
\begin{equation}\label{postpodm}
\begin{aligned}
\int\limits_{0}^{x}\frac{g_{\delta,\lambda}(t)-g_{\delta,\lambda}(0)}{t^2}\;dt&=
-\frac{\lambda}{\delta}\int\limits_{0}^{\delta}\frac{t^2}{t^2}\;dt+\int\limits_{\delta}^{\lambda}\frac{-2\lambda t+\lambda\delta}{t^2}\;dt
+\int\limits_{\lambda}^{x}\frac{-t^2+\lambda\delta-\lambda^2}{t^2}\;dt\\
&=-\lambda-2\lambda\log\left(\frac{\lambda}{\delta}\right)+(\lambda-\delta)-(x-\lambda)+\frac{(x-\lambda)(\delta-\lambda)}{x}\\
&\geq-\lambda-2\lambda\log\left(\frac{\lambda}{\delta}\right)+(\lambda-\delta)-(x-\lambda)+(\delta-\lambda)\\
&=-2\lambda\log\left(\frac{\lambda}{\delta}\right)-x.
\end{aligned}
\end{equation}

If we additionally assume that
\begin{equation}\label{deltalambda}
\delta=\frac{\lambda}{e^{\frac{1}{\lambda}}}\quad\text{which implies}\quad \log\left(\frac{\lambda}{\delta}\right)=\frac{1}{\lambda}
\end{equation}
we obtain
\begin{equation}
\int\limits_{0}^{x}\frac{g_{\delta,\lambda}(t)-g_{\delta,\lambda}(0)}{t^2}\;dt\geq-2-x,
\end{equation}
assuming $x\geq \lambda$.
Also,
\begin{equation}\label{lambda}
-\int\limits_{0}^{\lambda}\frac{g_{\delta,\lambda}(t)-g_{\delta,\lambda}(0)}{t^2}\;dt=2+\delta\geq 2.
\end{equation}
This in particular means, using Lemma~\ref{extension}, that there is a separately convex function $f_{\delta,\lambda}$ on $\er^2$ 
such that $f_{\delta,\lambda}(t,t)=g_{\delta,\lambda}(t)$ for every $t\in\er$.
Moreover, it will have the property that
\begin{equation}\label{horniodhad}
f_{\delta,\lambda}(u,v)\leq f_{\delta,\lambda}(0,0)-x\int\limits_{0}^{x}\frac{g_{\delta,\lambda}(t)-g_{\delta,\lambda}(0)}{t^2}\;dt\leq 1+2x+x^2
\end{equation}
whenever $|u|\leq x$, $|v|\leq x$ and $\lambda\leq x$.

Fix sequences $\{\delta_n\}$ and  $\{\lambda_n\}$ satisfying \eqref{deltalambda} and a sequence $\{a_n\}\subset[0,1]$, $a_n\searrow 0$, satisfying $a_n-\lambda_n>a_{n+1}+\lambda_{n+1}$.
Define functions
$f_n:\er^2\to\er$, $g_n:\er\to\er$, $n\in\en$, $f:\er^2\to\er$ and $g:\er\to\er$ by (cf. Figure~\ref{gecka}(b))
$$
g_n(x):=g_{\delta_n,\lambda_n}(x-a_n)-2a_nx+a_n^2,
$$

$$
f_n(x,y):=f_{\delta_n,\lambda_n}(x-a_n,y-a_n)-a_n(x+y)+a_n^2,
$$

$$
g(x):=\sup_n g_n(x)\quad\text{and}\quad f(x,y):=\sup_n f_n(x,y).
$$
Note that $f_n$ is a separately convex function such that $f_n(t,t)=g_n(t)$, $t\in\er$, which then implies $f(t,t)=g(t)$, $t\in\er$.
Also $g_n(x)=-x^2$ for $x\not\in[a_n-\lambda_n,a_n+\lambda_n]=:I_n$ and, since $I_n\cap I_m=\emptyset$ for $n\neq m$, $g$ is concave (because it is locally concave due to the concavity of every $g_n$).

We need to show that $f$ is finite at every $(u,v)\in\er^2$, but this is easy since using \eqref{horniodhad} we obtain
$$
\begin{aligned}
f_n(u,v)&=f_{\delta_n,\lambda_n}(u-a_n,v-a_n)-a_n(u+v)+a_n^2 \\
 &\leq 1+2x+x^2+2a_n(x-a_n)+a_n^2\leq 2+4x+x^2,
\end{aligned}
$$
provided $|u-a_n|\leq x,|v-a_n|\leq x$ and $\lambda_{n}\leq x$.

It remains to prove that $g$ is not semi-convex with any modulus $\omega$ satisfying
\begin{equation}\label{modulpodm}
\int_0^1\frac{\omega(t)}{t}\;dt<\infty.
\end{equation}

So suppose that $g$ is semi-convex with such a modulus $\omega$.
This gives us
$$
g(\alpha x+(1-\alpha)y)\leq \alpha g(x)+(1-\alpha)g(y)+\alpha(1-\alpha)|x-y|\omega(|x-y|)
$$
which can be for $\alpha=\frac{1}{2}$ rewritten as
\begin{equation}
-\frac{g(x)+g(y)-2g\left(\frac{x+y}{2}\right)}{|x-y|^2}\leq \frac{\omega(|x-y|)}{2|x-y|}.
\end{equation}
Considering $x=a_n+t$ and $y=a_n-t$ we then obtain
\begin{equation}
-\frac{g(a_n+t)+g(a_n-t)-2g(a_n)}{4t^2}\leq \frac{\omega(2t)}{4t}.
\end{equation}
This implies that, using \eqref{lambda},
$$
\begin{aligned}
\int_0^{2\lambda_n}\frac{\omega(t)}{t}\;dt&=2\int_0^{\lambda_n}\frac{\omega(2t)}{2t}\;dt\\
&\geq -\int_0^{\lambda_n}\frac{g(a_n+t)+g(a_n-t)-2g(a_n)}{t^2}\;dt\geq 2+2
\end{aligned}
$$
for every $n$, which is a contradiction with \eqref{modulpodm}.
\end{example}

\end{document}